\documentclass[final]{amsart}
\usepackage{amsmath}
  \usepackage{paralist}
  \usepackage{graphics} 
  \usepackage{epsfig} 
\usepackage{graphicx}  \usepackage{epstopdf}
 \usepackage[colorlinks=true]{hyperref}
\hypersetup{urlcolor=blue, citecolor=red}

\newtheorem{theorem}{Theorem}[section]

\newtheorem{lemma}[theorem]{Lemma}
\newtheorem{proposition}{Proposition}

\theoremstyle{definition}
\newtheorem{definition}[theorem]{Definition}
\newtheorem{remark}{Remark}
\newtheorem*{notation}{Notation}

\title[Plasma-vacuum interface] 
      {Stability of the linearized MHD-Maxwell free interface problem}

\author[D. Catania, M. D'Abbicco and P. Secchi]{}


 \email{davide.catania@unibs.it}
 \email{marcello.dabbicco@unibs.it}
 \email{paolo.secchi@unibs.it}

\newtheorem{hypothesis}[theorem]{Hypothesis}

\usepackage{amssymb}

\newcommand{\R}{{\mathbb R}}

\def\duno{\partial_1}

\def\dt{\partial_t}

\def\G{\mathcal G }
\def\H{\mathcal H }

\def\ds{\displaystyle}

\newcommand{\Hc}{{\H}} 

\newcommand{\p}{\partial}

\newcommand{\hg}{{\mathfrak{h}}}
\newcommand{\Hg}{{\mathfrak{H}}} 
\newcommand{\Ec}{{\mathcal{E}}} 
\newcommand{\eg}{{\mathfrak{e}}}
\newcommand{\Eg}{{\mathfrak{E}}} 
\newcommand{\Hs}{{\mathsf{H}}} 
\newcommand{\Es}{{\mathsf{E}}} 
\newcommand{\Wc}{{\mathcal{W}}} 
\newcommand{\Wg}{{\mathfrak{W}}} 
\newcommand{\wg}{{\mathfrak{w}}} 
\newcommand{\Ws}{{\mathsf{W}}}
\newcommand{\Uc}{\mathcal U}

\renewcommand{\G}{{\mathfrak{G}}}

\DeclareMathOperator{\dv}{{div}}
\DeclareMathOperator{\rot}{{\nabla \times\, }}
\newcommand{\tm}{^\intercal}
\renewcommand*{\doteq }{:=} 

\newcommand{\mep}{\varepsilon}

\begin{document}
\maketitle

\centerline{\scshape Davide Catania, Marcello D'Abbicco and Paolo Secchi}
\medskip
{\footnotesize
 \centerline{DICATAM, Mathematical Division, University of Brescia }
   \centerline{Via Valotti, 9, 25133 Brescia, Italy}
} 

%

\bigskip


\begin{abstract}
We consider the free boundary problem for the plasma-vacuum interface in ideal compressible magnetohydrodynamics (MHD). In the plasma region, the flow is governed by the usual compressible MHD equations, while in the vacuum region we consider the Maxwell system for the electric and the magnetic fields, in order to investigate  the well-posedness of the problem, in particular in relation with the electric field in vacuum. At the free interface, driven by the plasma velocity, the total pressure is continuous and the magnetic field on both sides is tangent to the boundary.

Under suitable stability conditions satisfied at each point of the plasma-vacuum interface, we derive a basic a priori estimate  for solutions to the  linearized problem in the Sobolev space $H^1_{\tan}$ with conormal regularity. The proof follows by a suitable secondary symmetrization of the Maxwell equations in vacuum and the energy method.

An interesting novelty is represented by the fact that the interface is characteristic with variable multiplicity, so that the problem requires a different number of boundary conditions, depending on the direction of the front velocity (plasma expansion into vacuum or viceversa). To overcome this difficulty, we recast the vacuum equations in terms of a new variable which makes the interface characteristic of constant multiplicity. In particular, we don't assume that plasma expands into vacuum.
\end{abstract}

\section{Introduction}
\label{sect1}

Plasma-vacuum interface problems appear in the mathematical modeling of plasma confinement by magnetic fields in thermonuclear energy production (as in Tokamaks; see, e.g., \cite{Goed}). In this model, the plasma is confined inside a perfectly conducting rigid wall and isolated from it by a region containing very low density plasma, which may qualify as vacuum, due to the effect of strong magnetic fields.
This subject is very popular since the 1950--70's, but most of theoretical studies are devoted to finding stability criteria of equilibrium states. The typical work in this direction is the classical paper of Bernstein et al. \cite{BFKK}.  In astrophysics, the plasma-vacuum interface problem can be used for modeling the motion of a star or the solar corona when magnetic fields are taken into account. Once again, the interface can be described as a tangential discontinuity, i.e. the magnetic fields do not intersect the interface.

In~\cite{SeTr,SeTrNl}, the authors studied the free boundary problem for the plasma-vacuum interface in ideal compressible magnetohydrodynamics (MHD), by considering the {\it pre-Maxwell dynamics} for the magnetic field in the vacuum region, as usually assumed in the classical formulation. The relativistic case has been addressed by Trakhinin in \cite{trakhinin12}, in the case of plasma expansion in vacuum. In this paper we consider again the non-relativistic case, but, instead of the pre-Maxwell dynamics, in the vacuum region we don't neglect the displacement current and consider the complete  system of {\it Maxwell equations} for the electric and the magnetic fields.

{The introduction of this model aims at investigating the influence of the electric field in vacuum on the well-posedness of the problem, as in the classical pre-Maxwell dynamics such an influence is hidden.
For the relativistic plasma-vacuum problem, Trakhinin \cite{trakhinin12} has shown the possible ill-posedness in the presence of a sufficiently strong vacuum electric field. Since relativistic effects play a rather passive role in the analysis of \cite{trakhinin12}, it is natural to expect the same for the non-relativistic problem under consideration here.
On the contrary, we will show that a {\it sufficiently weak} vacuum electric field, under the stability condition \eqref{41}, precludes ill-posedness and gives the well-posedness of the linearized problem (see Theorem~\ref{main}), thus somehow justifying the practice of neglecting the displacement current in the classical pre-Maxwell formulation when the vacuum electric field is weak enough. }

 {We discuss different equivalent formulations of the problem and study the stability} of the linearized problem with variable coefficients for nonplanar plasma-vacuum interfaces.
The main result is given in Theorem \ref{main}, Section \ref{mainresult}, where
we derive a basic a priori estimate  for solutions to the  linearized problem in the Sobolev space $H^1_{\tan}$ with conormal regularity, under suitable stability conditions satisfied at each point of the plasma-vacuum interface.
{The proof follows by a suitable secondary symmetrization of the Maxwell equations in vacuum and the energy method. The approach is similar to that in \cite{SeTr,trakhinin12}. }

{Observe that
one important difficulty of the plasma-vacuum problem, as for other free-boundary problems in MHD, is that we cannot test the Kreiss--Lopatinski condition \cite{kreiss, benzoni-serre, nappes} analytically. On the other hand, since the number of dimensionless parameters for the constant coefficients linearized problem is big, a complete numerical test of the Kreiss--Lopatinski condition seems unrealizable in practice. }

{The free interface, which is an unknown of the problem, turns out to be nonuniformly characteristic, that is characteristic of variable multiplicity, so that a different number of boundary conditions, depending on the direction of the front velocity (plasma expansion into vacuum or viceversa), are needed in different parts of the boundary. To overcome this huge obstacle, we introduce the new vacuum variable $\Ws$, defined in \eqref{defss}, which makes the interface characteristic of constant multiplicity, with no need of further assumptions such as plasma expansion (as in \cite{trakhinin12,mandrik-trakhinin}). This can be considered the main novelty of the paper.}

\medskip

First we consider the equations of ideal compressible MHD in the plasma region:
\begin{equation}
\left\{
\begin{array}{l}
\partial_t\rho  +{\rm div}\, (\rho {v} )=0\, ,\\[3pt]
\partial_t(\rho {v} ) +{\rm div}\,(\rho{v}\otimes{v} -{H}\otimes{H} ) +
{\nabla}q=0\, , \\[3pt]
\partial_t{H} -{\nabla}\times ({v} {\times}{H})=0\, ,\\[3pt]
\partial_t\bigl( \rho e +\frac{1}{2}|{H}|^2\bigr)+
{\rm div}\, \bigl((\rho e +p){v} +{H}{\times}({v}{\times}{H})\bigr)=0\, ,
\end{array}
\right.
\label{MHD}
\end{equation}
where $\rho$ denotes density, $v\in\mathbb{R}^3$ plasma velocity, $H \in\mathbb{R}^3$ magnetic field, $p=p(\rho,S )$ pressure, $q =p+\frac{1}{2}|{H} |^2$ total pressure, $S$ entropy, $e=E+\frac{1}{2}|{v}|^2$ total energy, and  $E=E(\rho,S )$ internal energy. With a state equation of gas, $\rho=\rho(p ,S)$, and the first principle of thermodynamics, \eqref{MHD} is a closed system.

System \eqref{MHD} is supplemented by the divergence constraint
\begin{equation}
{\rm div}\, {H} =0
\label{2}
\end{equation}
on the initial data. As it is known, taking into account \eqref{2}, we can easily symmetrize system \eqref{MHD} by rewriting it in the nonconservative form
\begin{equation}
\left\{
\begin{array}{l}
{\displaystyle\frac{\rho_p}{\rho}}\,{\displaystyle\frac{{\rm d} p}{{\rm d}t} +{\rm div}\,{v} =0}\, ,\qquad
\rho\, {\displaystyle\frac{{\rm d}v}{{\rm d}t}-({H}\cdot\nabla ){H}+{\nabla}
q  =0 }\, ,\\[9pt]
{\displaystyle\frac{{\rm d}{H}}{{\rm d}t} - ({H} \cdot\nabla ){v} +
{H}\,{\rm div}\,{v}=0}\, ,\qquad
{\displaystyle\frac{{\rm d} S}{{\rm d} t} =0}\, ,
\end{array}\right. \label{MHDnon}
\end{equation}
where $\rho_p\equiv\partial\rho/\partial p$ and ${\rm d} /{\rm d} t =\partial_t+({v} \cdot{\nabla} )$.
A different symmetrization is obtained if we consider $q$ instead of $p$.
In terms of $q $, the equation for the pressure in \eqref{MHDnon} takes
the form
\begin{equation}
\begin{array}{ll}\label{equq}
\displaystyle\frac{\rho_p}{\rho }\left\{\frac{{\rm d} q}{{\rm d}t} -H
\cdot\displaystyle\frac{{\rm d} H}{{\rm d}t} \right\}+{\rm div}\,{v}=0\, ,
\end{array}
\end{equation}
where it is understood that now
$\rho =\rho(q  -|H  |^2/2,S)$ and similarly for
$\rho_p $.
Then we derive ${\rm div}\,{v} $ from \eqref{equq} and rewrite the equation
for the magnetic field in \eqref{MHDnon} as
\begin{equation}
\begin{array}{ll}\label{equH}
\displaystyle\frac{{\rm d} H}{{\rm d}t} -(H \cdot\nabla)v  -
 \frac{\rho_p}{\rho }H\left\{\frac{{\rm d} q}{{\rm d}t} -H
\cdot\frac{{\rm d} H}{{\rm d}t} \right\}=0.
\end{array}
\end{equation}
Substituting \eqref{equq}, \eqref{equH} in \eqref{MHDnon} then gives the following symmetric system
\begin{equation}
\begin{array}{ll}\label{3'}
\left(\begin{matrix}
{\rho_p/\rho}&\underline
0&-({\rho_p/\rho})H &0 \\
\underline 0^T&\rho
I_3&0_3&\underline 0^T\\
-({\rho_p/\rho})H^T&0_3&I_3+({\rho_p/\rho})H\otimes H&\underline 0^T\\
0&\underline 0&\underline 0&1
\end{matrix}\right)\dt
\left(\begin{matrix}
q \\ v \\
H\\S \end{matrix}\right)+\\
 \\
 +
\left( \begin{matrix}
(\rho_p/\rho)
v \cdot\nabla&\nabla\cdot&-({\rho_p/\rho})Hv \cdot\nabla&0\\
\nabla&\rho v \cdot\nabla I_3&-H \cdot\nabla I_3&\underline 0^T\\
-({\rho_p/\rho})H^T v \cdot\nabla&-H \cdot\nabla I_3&(I_3+({\rho_p/\rho})H\otimes H)
v \cdot\nabla&\underline 0^T\\
0&\underline 0&\underline 0&v \cdot\nabla
\end{matrix}\right)
\left(\begin{matrix}q \\ v \\ H\\S \end{matrix}\right)
=0\,,
\end{array}
\end{equation}
where $\underline 0=(0,0,0)$.
Given this symmetrization, as the unknown we can choose the vector $ U =U (t, x )=(q, v,H, S)$. For the sake of brevity we write system (\ref{3'}) in the form
\begin{equation}
\label{4}
A_0(U )\partial_tU+\sum_{j=1}^3A_j(U )\partial_jU=0\, ,
\end{equation}
which is symmetric hyperbolic provided the hyperbolicity condition $A_0>0$ holds, i.e.
\begin{equation}
\rho  >0\, ,\qquad \rho_p >0\, . \label{5}
\end{equation}
We use $\p_j$, $j=1,2,3$, to denote the partial derivative with respect to~$x_j$; in the sequel we will use $\p_0=\p_t$ to denote the partial derivative with respect to~$t$ (see Appendix~\ref{App:Notation}).

Let $\Omega^+(t)$ and $\Omega ^-(t)$ be space-time domains occupied by the plasma and the vacuum respectively. That is, in the domain
$\Omega^+(t)$ we consider system \eqref{4} governing the motion of an ideal plasma and in the domain $\Omega^-(t)$  we have the Maxwell system
\begin{equation}\label{eq:Maxwell}
\begin{cases}
\mep\p_t \Hc + \rot \Ec = 0  \,, \\
\mep\p_t \Ec - \rot \Hc = 0  \,,
\end{cases}
\end{equation}
describing the vacuum magnetic and electric fields $\Hc,\Ec\in\mathbb{R}^3$. {Here, the equations are written in nondimensional form through a suitable scaling (see Mandrik--Trakhinin~\cite{mandrik-trakhinin}), and $\mep=\frac{\bar{v}}{c}$, where $\bar{v}$ is the velocity of a uniform flow and $c$ is the speed of light in vacuum. If we choose $\bar{v}$ to be the speed of sound in vacuum, we have that $\mep$ is a small, even though fixed parameter.}

System \eqref{eq:Maxwell} is supplemented by the divergence constraints
\begin{equation*}
 \dv \Hc = \dv \Ec = 0
\label{}
\end{equation*}
on the initial data.
We write system~\eqref{eq:Maxwell} in the form:
\begin{equation}\label{eq:MaxwellB}
\mep\p_t \Wc + \sum_{j=1}^3 B_j\p_j \Wc = 0\,,
\end{equation}
where~$\Wc=(\Hc,\Ec)\tm$, and
\begin{gather} B_j = \begin{pmatrix}
O_3 & B'_j \\
{B'_j}\tm & O_3
\end{pmatrix} ,\ j=1,2,3, \\
B'_1=\begin{pmatrix}
0 & 0 & 0 \\
0 & 0 & -1 \\
0 & 1 & 0
\end{pmatrix}, \
B'_2=\begin{pmatrix}
0 & 0 & 1 \\
0 & 0 & 0 \\
-1 & 0 & 0
\end{pmatrix}, \
B'_3=\begin{pmatrix}
0 & -1 & 0 \\
1 & 0 & 0 \\
0 & 0 & 0
\end{pmatrix}. \label{eq:Bmatrix}
\end{gather}

Let us assume that the interface between plasma and vacuum is given by a hypersurface $\Gamma (t)=\{F(t,x)=0\}$. It has to be determined and moves with the velocity of plasma particles at the boundary:
\begin{equation}
\frac{{\rm d}F }{{\rm d} t}=0\quad \mbox{on}\ \Gamma (t)\label{7}
\end{equation}
(for all $t\in [0,T]$). Since $F$ is an unknown of the problem, this is a free-boundary problem. The plasma variable $U$ is connected with the vacuum magnetic and electric fields
$\mathcal{H}, \Ec$  through the relations \cite{BFKK,Goed}
\begin{equation}
 [q]=0,\quad  H\cdot N=0, \quad \Hc\cdot N=0 ,\quad  N \times \Ec = \mep (v \cdot N)  \Hc\quad \mbox{on}\ \Gamma (t),
\label{8}
\end{equation}
where $N=\nabla F$ and $[q]= q|_{\Gamma}-\frac{1}{2}|\mathcal{H}|^2_{|\Gamma}+\frac12|\Ec|^2_{|\Gamma}$ denotes the jump of the total pressure across the interface.
These relations together with \eqref{7} are the boundary conditions at the interface $\Gamma (t)$. Let us note that, in particular, the magnetic fields on both sides are tangent to the free interface.

As in \cite{lindblad2,trakhinin09, SeTr, SeTrNl}, we will assume that for problem \eqref{4}, \eqref{eq:MaxwellB}--\eqref{8} the hyperbolicity conditions \eqref{5} hold true in $\Omega^+(t)$ up to the boundary $\Gamma(t)$, i.e., the plasma density does not go to zero continuously, but has a jump (clearly in the vacuum region $\Omega^-(t)$ the density is identically zero). This assumption is compatible with the continuity of the total pressure in \eqref{8}.
For instance, in the case of ideal polytropic gases one has $p=A\rho^\gamma e^S$ with $A>0,\gamma>1$. Then the continuity of the total pressure at $\Gamma$ requires $(A\rho^\gamma e^S+\frac{1}{2}|{H}|^2){_{|\Gamma^+}=(\frac{1}{2}|\mathcal{H}|^2-\frac{1}{2}|\Ec|^2)_{|\Gamma^-}}$, which may be obtained also for densities $\rho$ discontinuous across $\Gamma$. Differently, in the absence of the magnetic field, the continuity of the pressure yields the continuity of the density so that $\rho_{|\Gamma_+}=0$.

Since the interface moves with the velocity of plasma particles at the boundary, by introducing the Lagrangian coordinates one can reduce the original problem to that in a fixed domain. This approach has been recently employed with success in a series of papers on the Euler equations in vacuum, see \cite{coutandshkoller3,coutandlindbladshkoller,coutandshkoller1,coutandshkoller2,lindblad2}. However, as for tangential discontinuities in various models of fluid dynamics (e.g. \cite{catcvs,cmst,morandotrakhinintrebeschi,mttpd,mttpv,trakhinin09arma}), this approach seems hardly applicable to problem \eqref{4}, \eqref{eq:MaxwellB}--\eqref{8}. Therefore, we will work in the Eulerian coordinates and for technical simplicity we will assume that the space-time domains $\Omega^\pm (t)$ have the following form.


Let us assume that the moving interface $\Gamma(t)$ takes the form $$
\Gamma(t) \doteq  \{ (x_1,x') \in \R^3 \, , \, x_1=\varphi(t,x')\} \, ,
$$
where $t \in [0,T]$ and $x'=(x_2,x_3)$. Then we have $\Omega^\pm(t)=\{x_1\gtrless \varphi(t,x')\}$.
With our parametrization of $\Gamma (t)$, an equivalent formulation
of the boundary conditions \eqref{7}, \eqref{8} at the interface is
\begin{equation}\label{eq:HN}
H_N=\mathcal{H}_N=0 \quad \mbox{on} \quad \Gamma(t)\, ,
\end{equation}
\begin{gather}
\label{eq:15-1}
\partial_t\varphi =v_N\, ,\quad [q]=0\, ,\\
\Ec_2+{\Ec_1\partial_2\varphi-\mep\Hc_3\partial_t \varphi}=0\, , \quad \Ec_3+{\Ec_1\partial_3\varphi+\mep\Hc_2\partial_t\varphi}=0 \quad \mbox{on}\quad \Gamma (t)\, ,
\label{15}
\end{gather}
where $v_N=v\cdot N$, $H_N=H\cdot N$, $\mathcal{H}_N=\mathcal{H}\cdot N$, $N=(1,-\partial_2\varphi ,-\partial_3\varphi )$. In particular, \eqref{15} consists of the third and second components of $N\times\Ec = {\mep\Hc\partial_t\varphi} $, whereas the first component of $N\times\Ec = {\mep\Hc\partial_t\varphi} $ follows as a direct consequence of~\eqref{15} and $\Hc_N|_{\Gamma(t)}=0$.

\medskip

The analysis of the problem shows that the number of boundary conditions that should be imposed to \eqref{4}, \eqref{eq:MaxwellB} varies with the sign of the front velocity $\partial_t\varphi$, see more details in Appendix \ref{calcolosegni} and \cite{trakhinin12}.
This number should correspond to the number of incoming characteristics for both \eqref{4} and \eqref{eq:MaxwellB}, plus one for the determination of the front. The correct number of boundary conditions is four if $\partial_t\varphi<0$ (the plasma expands into vacuum). Thus the four boundary conditions \eqref{eq:15-1}, \eqref{15} are in agreement with this necessary condition for well-posedness. (It will be shown that the conditions in \eqref{eq:HN} may be considered just as a restriction on the initial data.) On the contrary, if $\partial_t\varphi>0$ (the vacuum expands into plasma) the correct number is six and problem \eqref{4}, \eqref{eq:MaxwellB}, \eqref{eq:15-1}, \eqref{15} is formally underdetermined because it is missing two boundary conditions. In view of that, let us denote
$$
\Gamma_\pm(t) \doteq  \{ (x_1,x') \in \R^3 \, , \, x_1=\varphi(t,x'),\, \partial_t\varphi(t,x')\gtrless0\} \, ,
$$
and supplement the system with the boundary conditions
\begin{equation}\label{eq:div}
 \dv \Hc = \dv \Ec = 0 \quad \mbox{on} \ \Gamma^+(t)
 \end{equation}
for all $t\in [0,T]$. These additional boundary conditions, {which are chosen according to \cite{trakhinin12},} are unnecessary on $\Gamma^-(t) $.

System \eqref{4}, \eqref{eq:Maxwell}, \eqref{eq:HN}--\eqref{eq:div} is supplemented with initial conditions
\begin{equation}
\begin{array}{ll}
{U} (0,{x})={U}_0({x})\, ,\quad {x}\in \Omega^{+} (0)\, ,\qquad
\varphi(0,{x})=\varphi_0({x})\qquad {x}\in\Gamma \, ,\\
\mathcal{H}(0,x)= \mathcal{H}_0(x)\, ,\quad \mathcal{E}(0,x)= \mathcal{E}_0(x) \qquad  {x}\in \Omega^{-}(0)\, .\label{11}
\end{array}
\end{equation}
{Due to the particular boundary conditions, problem \eqref{4}, \eqref{eq:MaxwellB}, \eqref{eq:HN}--\eqref{11} is an initial boundary value problem with characteristic boundary. Moreover, since we prescribe a different number of conditions on different portions of the boundary (see \eqref{eq:div}), the boundary is nonuniformly characteristic, that is characteristic with variable multiplicity. A satisfactory theory for such nonuniformly characteristic boundary value problems is still lacking, see \cite{nishitani00,secchinonunif1,secchinonunif2} and references therein included.}


From the mathematical point of view, a natural wish is to find conditions on the initial data
providing the existence and uniqueness on some time interval $[0,T]$ of a solution $({U},\Hc,\Ec,\varphi)$ to problem \eqref{4}, \eqref{eq:MaxwellB}, \eqref{eq:HN}--\eqref{11} in Sobolev spaces. Since \eqref{4} is a system of nonlinear hyperbolic conservation laws that can produce shock waves and other types of strong discontinuities
(e.g., current-vortex sheets \cite{chenwang,trakhinin09arma}), it is natural to expect to obtain only local-in-time existence theorems.
Morever, it is natural to expect in general a local-in time existence for problem \eqref{4}, \eqref{eq:MaxwellB}, \eqref{eq:HN}--\eqref{11}, even under suitable stability conditions for well-posedness, because of the occurrence of instabilities coming from the interface in finite time.

We must regard the boundary conditions on $H, \Hc$ in \eqref{eq:HN} as the restriction on the initial data \eqref{11}. Similarly for $\dv H=\dv \Hc=\dv\Ec=0$.
\\
More precisely, we can prove that a solution of~\eqref{4}, \eqref{eq:15-1} (if it exists for all $t\in [0,T]$) satisfies
\[
 \dv{H} =0 \quad \mbox{in}\ \Omega^+ (t)\quad \mbox{and}\quad  H_N=0\quad \mbox{on}\ \Gamma (t),
\]
for all $t\in [0,T]$, if the latter were satisfied at $t=0$, i.e., for the initial data \eqref{11}. Similarly, we can prove that a solution of \eqref{eq:Maxwell}, \eqref{15}, \eqref{eq:div} satisfies
\[ \dv \Hc = \dv \Ec = 0 \quad \mbox{in} \ \Omega^-(t) , \quad \Hc_N=0 \quad {\mbox{on}\ \Gamma(t)},
\]
for all $t\in [0,T]$, if the latter were satisfied at $t=0$, i.e., for the initial data \eqref{11}.

The content of the paper is as follows: the rest of this introduction is devoted to an equivalent formulation of the free boundary problem \eqref{4}, \eqref{eq:MaxwellB}, \eqref{eq:15-1}--\eqref{11} on the fixed
domain with flat boundary. In Section \ref{sec:linearization} we introduce the linearized problem with some useful reductions and equivalent formulations. After introducing some function spaces in Section \ref{fs}, we present our main result in Section \ref{mainresult}, Theorem \ref{main}. The proof of the main result is given in Section \ref{basic}. {Appendix~\ref{App:Notation} collects the main notations, hoping that it may help the reader, Appendix~\ref{computations} contains the proofs of some technical lemmas, and in Appendix~\ref{calcolosegni} we discuss the number of boundary conditions.}

\subsection{An equivalent formulation in the fixed domain} \label{subsecfixed}

Let us denote
$$
\Omega^\pm \doteq  \R^3 \cap \{x_1\gtrless 0 \} \, ,\qquad \Gamma \doteq \R^3\cap\{x_1=0\} \, .
$$
We want to reduce the free boundary problem \eqref{4}, \eqref{eq:MaxwellB}, \eqref{eq:15-1}--\eqref{11} to the fixed
domains $\Omega^\pm$. For this purpose we introduce a suitable change of coordinates that is inspired by Lannes \cite{lannes} (see also \cite{cmst,SeTr}). In all what follows, $H^s(\omega)$ denotes the Sobolev space of order $s$ on a domain $\omega$.

\begin{notation}\label{not:derivatives}
We use the notation~$\Psi_j$ to denote~$\p_j \Psi$, being~$\Psi$ a scalar function, whereas by $\Hc_j, \Ec_j$ we denote the~$j$-th component of $\Hc, \Ec$, being them vector functions (see Appendix~\ref{App:Notation}).
\end{notation}

The diffeomorphism that reduces the free boundary problem \eqref{4}, \eqref{eq:MaxwellB}, \eqref{eq:15-1}--\eqref{11} to the fixed domains $\Omega^\pm$ is given in the following lemma.
\begin{lemma}[Lemma 3 in~\cite{SeTr}]
\label{lemma3}
Let $m \ge3$ be an integer. For any $T>0$, and
for any
\[ \varphi \in \cap_{j=0}^{m-1} {\mathcal C}^j([0,T];H^{m-j-\frac12}(\R^2))\,,\]
satisfying without loss of generality $\| \varphi\|_{{\mathcal C}([0,T];H^{2}(\R^2))} \le 1$, there exists a function \[\Psi \in \cap_{j=0}^{m-1} {\mathcal C}^j([0,T];H^{m-j}(\R^3))\]
such that the function
\begin{equation}
\label{change}
\Phi(t,x) \doteq  \big( x_1 +\Psi(t,x),x' \big) \, , \qquad (t,x) \in [0,T]\times \R^3 \, ,
\end{equation}
defines an $H^m$-diffeomorphism of $\R^3$ for all $t \in [0,T]$.
Moreover, there holds \[\partial^j_t (\Phi - Id) \in {\mathcal C}([0,T];H^{m-j}(\R^3))\] for $j=0,\dots, m-1$,
$\Phi(t,0,x')=(\varphi(t,x'),x')$, $\duno \Phi(t,0,x')=(1,0,0)$, as well as
\begin{gather}
\label{eq:Phitcontrol}
\|\p_t^j \Phi(t,\cdot)\|_{L^\infty(\R^3)} \leq \frac1{\sqrt{2\pi}}\,
\|\p_t^j \varphi(t,\cdot)\|_{H^{\frac32}(\R^2)} \qquad t\in[0,T] \,,   \\
\label{eq:Phiest}
\| \Psi_1(t,\cdot)\|_{L^\infty(\R^3)} \leq \frac{1}{2} \qquad t\in[0,T]\,,
\end{gather}
for $j=1,\ldots,m-{2}$.
\end{lemma}
In particular, from~\eqref{eq:Phiest} we derive $1+\Psi_1(t,x) \geq 1/2$. For the proof of Lemma~\ref{lemma3}, see Appendix~\ref{computations}.

\medskip

We introduce the change of independent variables defined by \eqref{change} by setting
\begin{gather*}
\tilde{U}(t,x )\doteq  {U}(t,\Phi (t,x))\, , \quad \tilde{\Hc}(t,x )\doteq  \Hc(t,\Phi (t,x))\, ,\\
\tilde{\Ec}(t,x)\doteq  \Ec(t,\Phi(t,x))\, , \quad \tilde{\Wc}(t,x )\doteq  \Wc(t,\Phi (t,x)) \, ,
\end{gather*}
and define the new dependent variables
\begin{align*}
\hg & =(\tilde\Hc_1-\Psi_2\tilde\Hc_2 -\Psi_3\tilde\Hc_3,(1+\Psi_1)\tilde\Hc_2,(1+\Psi_1)\tilde\Hc_3)\tm, \\
\eg & =(\tilde\Ec_1-\Psi_2\tilde\Ec_2-\Psi_3\tilde\Ec_3,(1+\Psi_1)\Ec_2,(1+\Psi_1)\Ec_3)\tm,\\
\Hs & =((1+\Psi_1)\tilde\Hc_1,\tilde\Hc_2+\Psi_2\tilde\Hc_1+\mep\Psi_t\tilde\Ec_3,
\tilde\Hc_3+\Psi_3\tilde\Hc_1-\mep\Psi_t\tilde\Ec_2)\tm,\\
\Es & =((1+\Psi_1)\tilde\Ec_1,\tilde\Ec_2+\Psi_2\tilde\Ec_1-\mep\Psi_t\tilde\Hc_3,\tilde\Ec_3
+\Psi_3\tilde\Ec_1+\mep\Psi_t\tilde\Hc_2)\tm. \\
\end{align*}
Moreover, we set
\begin{align}\label{defss}
\wg = (\hg, \eg)\tm \, , \qquad \Ws = (\Hs, \Es)\tm\, .
\end{align}
Notice that

Let us define the matrix
\[ \eta \doteq  \begin{pmatrix}
1 & - \Psi_2 & - \Psi_3 \\
0 & 1+\Psi_1 & 0 \\
0 & 0 & 1+\Psi_1
\end{pmatrix} , \]
which is invertible by virtue of the smallness of $\Psi_1$ (see Lemma~\ref{lemma3}). Then:
\[ \hg=\eta\,\tilde\Hc\, , \qquad \eg=\eta\,\tilde\Ec\, ,\qquad \wg = K\tilde{\Wc}\, , \]
where
\begin{equation}\label{eq:K}
K = \begin{pmatrix}
\eta & 0 \\
0 & \eta
\end{pmatrix} .
\end{equation}
We notice that $\Ws = J \,\tilde\Wc$, where
\begin{align}\label{eq:J} J &\doteq  \begin{pmatrix}
(1+\Psi_1)(\eta\tm)^{-1} & -\mep\Psi_t B_1'  \\
\mep\Psi_t B_1' & (1+\Psi_1)(\eta\tm)^{-1}
\end{pmatrix} \\
&\equiv \begin{pmatrix}
1+\Psi_1 & 0 & 0 & 0 & 0 & 0 \\
\Psi_2 & 1 & 0 & 0 & 0 & \mep\Psi_t \\
\Psi_3 & 0 & 1 & 0 & -\mep\Psi_t & 0 \\
0 & 0 & 0 & 1+\Psi_1 & 0 & 0  \\
0 & 0 & -\mep\Psi_t & \Psi_2 & 1 & 0  \\
0 & \mep\Psi_t & 0 & \Psi_3 & 0 & 1
\end{pmatrix} \,,
\end{align}
with~$B_1'$ as in~\eqref{eq:Bmatrix}. The matrix~$J$ is not invertible, in general, since
\[ \det J = (1+\Psi_1)^2 (1-{\mep^2}\Psi_t^2)^2\,, \]
and~$1-{\mep^2}\Psi_t^2$ might vanish. {However, this is prevented if $\mep$ is sufficiently small, which is true for physical problems. Otherwise, this difficulty can be overcome} by assuming the smallness of the normal velocity of the plasma at the boundary.

{Notice that the relation $\Ws=J\tilde \Wc$ corresponds to the non-relativistic version of the Joule--Bernoulli equation connecting the magnetic and the electric fields in two inertial frames moving at relative velocity equal to the interface speed when the interface curvature is neglected.}

\begin{hypothesis}\label{hyp:invertibility0}
We assume that~$\mep v_N$ is sufficiently small on~$\Gamma$, namely
\begin{equation}\label{eq:smallvn0}
\frac1{\sqrt{2\pi}}\,\|\mep v_N(t,0,\cdot)\|_{H^{\frac32}(\R^2)}< 1\,.
\end{equation}
\end{hypothesis}
By virtue of~\eqref{eq:15-1}, \eqref{eq:Phitcontrol}, {from Hypothesis~\ref{hyp:invertibility0}} it follows that $\sup |\mep\Psi_t(t,x)|<1$ in $[0,T]\times\R^3$, and~$J$ becomes invertible. Thus we can recover ${\tilde\Wc}$ from $\Ws$, by ${\tilde\Wc}=J^{-1}\Ws$. We remark that~$\Psi_t(t,x_1,x')=0$ for~$|x_1|$ sufficiently large, namely $x_1\not\in \mathrm{supp}\, \chi$ (see Appendix \ref{computations} for the definition of $\chi$), hence it is not possible to derive the invertibility of~$J$ by assuming $|\mep\Psi_t(t,x)|>1$ for all~$(t,x)\in[0,T]\times\R^3$.
\medskip

For the reader's convenience, we define
\begin{equation}\label{eq:J1}
J_1 \doteq  (1+\Psi_1)(1-{\mep^2}\Psi_t^2)J^{-1} \,,
\end{equation}
whose explicit form is given in Appendix~\ref{App:Notation}.

In~\cite{SeTr,SeTrNl}, assuming \emph{pre-Maxwell} dynamics in vacuum, i.e. $\rot \Hc=0,\, \dv\Hc=0$ in~$\Omega^-(t)$, the authors obtained the equation~$\rot \Hg=0$ in the fixed domain, with the variable~$\Hg$ given by~$\Hg\doteq (1+\Psi_1)(\eta\tm)^{-1}\tilde\Hc$. They also made use of the variable~$\Eg\doteq (1+\Psi_1)(\eta\tm)^{-1}\tilde\Ec$. Therefore,
\[ \tilde\Wc = \frac1{1+\Psi_1}\,K\tm \Wg\,, \qquad \Wg=(\Hg,\Eg)\tm \,. \]
Having this in mind, one may also write
\[ \Ws = \Wg + \mep \Psi_t \begin{pmatrix}
0 & - B_1'  \\
B_1' & 0
\end{pmatrix} \tilde\Wc \equiv \Wg + \mep\Psi_t \left( 0 ,\tilde\Ec_3 ,-\tilde\Ec_2,0 ,-\tilde\Hc_3 ,\tilde\Hc_2\right)\tm \,,\]
as well as
\[ \Ws=L\Wg\,, \qquad L\doteq  \frac1{1+\Psi_1}\,J\,K\tm \,. \]
The variable~$\Wg$ is also used in~\cite{trakhinin12}.

\medskip

{By resorting to the variable $\tilde\Wc$, the equations \eqref{eq:MaxwellB} in $\Omega^-(t)$ can be recast on the fixed domain $\Omega^-$ as
\begin{align} \label{sys.Wtilde}
\mep\p_t \tilde\Wc + \frac1{1+\Psi_1}\,\Bigl( B_1 - \mep\Psi_t I-\sum_{j=2,3} \Psi_j\,B_j \Bigr)\p_1 \tilde\Wc + \sum_{j=2,3} B_j \p_j \tilde\Wc = 0\, .
\end{align}}

The new set of dependent variables $\wg, \Ws$ defined in \eqref{defss} is convenient for the reformulation of equations on the fixed domain~$\Omega^-$.
\begin{proposition}\label{prop:Maxwell}
System \eqref{eq:Maxwell} (or \eqref{eq:MaxwellB}) in $\Omega^-(t)$, {or \eqref{sys.Wtilde} in $\Omega^-$,} is equivalent to
\begin{equation} \label{eq:Maxwell1}\begin{cases}
\mep\p_t \hg + \rot \Es + \mep Q\dv\hg= 0\,,\\
\mep\p_t \eg - \rot \Hs + \mep Q\dv\eg= 0\end{cases} \end{equation}
in the fixed domain~$\Omega^-$, where
\begin{equation*}
Q=- \frac{\Psi_t}{1+\Psi_1}\begin{pmatrix} 1 \\ 0 \\ 0 \end{pmatrix}
\end{equation*}
Equations~$\dv\Hc=0$ and $\dv\Ec=0$ in $\Omega^-(t)$ are equivalent to
\begin{equation} \label{eq:Maxwell2}\begin{cases}
\dv \hg = 0 \,,\\
\dv \eg = 0
\end{cases} \end{equation}
in the fixed domain~$\Omega^-$.
\end{proposition}
For the proof of Proposition~\ref{prop:Maxwell}, see Appendix~\ref{computations}.
Proposition~\ref{prop:Maxwell} implies that the additional boundary conditions \eqref{eq:div} are equivalent to
\begin{equation*}
 \dv \hg = \dv \eg = 0 \quad \mbox{on} \ \Gamma\cap \{ \partial_t\varphi(t,x')>0\} \, ,
 \end{equation*}
for all $t\in [0,T]$.
\begin{remark}\label{rem:conditions1}
The two boundary conditions in~\eqref{15} can be written as
\begin{align*}
\Es_2 (t,0,x') &= (\tilde\Ec_2+\tilde\varphi_2\tilde\Ec_1-\mep\tilde\varphi_t\tilde\Hc_3)(t,0,x') \\
&= (\Ec_2+\varphi_2\Ec_1-\mep\varphi_t\Hc_3)\circ (t,\varphi(t,x'),x') = 0\,, \\
\Es_3 (t,0,x') &= (\tilde\Ec_3+\tilde\varphi_3\tilde\Ec_1+\mep\tilde\varphi_t\tilde\Hc_2)(t,0,x') \\
&= ({\Ec_3}+\varphi_3\Ec_1+\mep\varphi_t\Hc_2)\circ (t,\varphi(t,x'),x') = 0\, ,
\end{align*}
where we used the trivial property $\tilde\varphi_j = (\varphi_j) \circ (t,\Phi)$, being~$\varphi_1\equiv0$.
\end{remark}


Dropping for convenience tildes in $\tilde{U}$ and $\tilde{\mathcal{W}}$, problem \eqref{4}, \eqref{eq:MaxwellB}, \eqref{eq:15-1}--\eqref{11} can be reformulated on the fixed
reference domains $\Omega^\pm$ as
\begin{equation}
\mathbb{P}(U,\Psi)=0\quad\mbox{in}\ [0,T]\times \Omega^+,\label{16}
\end{equation}
\begin{equation}
{\tilde{\mathbb{V}}}(\Wc,\Psi)=0
\quad\mbox{in}\ [0,T]\times \Omega^-,\label{16'}
\end{equation}
\begin{equation}
\mathbb{B}(U,\Wc,\varphi )=0\quad\mbox{on}\ [0,T] \times\Gamma,\label{17}
\end{equation}
\begin{equation}\label{eq:div2}
 \dv \hg = \dv \eg = 0 \quad \mbox{on} \ [0,T] \times\Gamma\cap \{ \partial_t\varphi(t,x')>0\} \, ,
 \end{equation}
\begin{equation}
(U,\Wc)|_{t=0}=(U_0,\Wc_0)\quad\mbox{in}\ \Omega^+\times\Omega^-,\qquad \varphi|_{t=0}=\varphi_0\quad \mbox{on}\ \Gamma,\label{18}
\end{equation}
where $\mathbb{P}(U,\Psi)=P(U,\Psi)U$,
\[
P(U,\Psi)=A_0(U)\partial_t +\tilde{A}_1(U,\Psi)\partial_1+A_2(U )\partial_2+A_3(U )\partial_3,
\]
\[
\tilde{A}_1(U,\Psi )=\frac{1}{1+\Psi_1}\Bigl(
A_1(U )-A_0(U)\Psi_t -\sum_{j=2}^3 A_j(U)\Psi_j \Bigr),
\]
\[
{\tilde{\mathbb{V}}}(\Wc,\Psi)=\mep\p_t \Wc + \tilde{B}_1 (\Psi) \p_1 \Wc + \sum_{j=2,3} B_j \p_j \Wc
\]
\[ \tilde{B}_1(\Psi) = \frac1{1+\Psi_1}\,\Bigl( B_1 - {\mep}\Psi_t I-\sum_{j=2,3} \Psi_j\,B_j \Bigr) , \]
\[
\mathbb{B}(U,\Wc,\varphi )=\left(
\begin{array}{c}
\partial_t\varphi -v_{N |x_1=0}\\ {[}q{]} \\
\Es_{2|{x_1=0}} \\
\Es_{3|{x_1=0}} \\
\end{array}
\right),\qquad v_{N}=v_1- v_2\Psi_2 - v_3\Psi_3
 ,
\]
\[
[q]=q_{|x_1=0}-\frac{1}{2}|\mathcal{H}|^2_{x_1=0}+\frac{1}{2}|\Ec|^2_{x_1=0}
\, .
\]

To avoid an overload of notation, we have denoted by the same symbol $v_N$ here above and $v_N$ as in \eqref{eq:15-1}. Notice that
$v_{N |x_1=0}=v_1- v_2\varphi_2 - v_3\varphi_3$, as in the previous definition in \eqref{eq:15-1}. Similarly, we will also denote $\Hc_N\doteq \hg_1$, $\Ec_N\doteq \eg_1$. We also define
\[ h \doteq  (H_1-\Psi_2 H_2-\Psi_3 H_3, (1+\Psi_1)H_2, (1+\Psi_1)H_3) \,, \qquad H_N\doteq  h_1\,. \]
Now we show that problem \eqref{16}--\eqref{18} implies the equations
\begin{equation}
\dv h=0\quad\mbox{in}\ [0,T]\times \Omega^+,\label{19}
\end{equation}
\begin{equation}
 \dv \hg =\dv \eg =0\quad\mbox{in}\ [0,T]\times \Omega^-,\label{eq:diveghg}
\end{equation}
and the boundary conditions
\begin{equation}
H_{N}=\Hc_{N}=0\quad\mbox{on}\ [0,T]\times\Gamma,\label{20}
\end{equation}
that can be considered just as restrictions on the initial data \eqref{18}.

\begin{proposition}
Let the initial data \eqref{18} satisfy \eqref{19}--\eqref{20} for $t=0$.
If $(U,\Wc,\varphi )$ is a regular solution  of problem \eqref{16}--\eqref{18}, then this solution satisfies \eqref{19}--\eqref{20} for all $t\in [0,T]$.
\label{p1}
\end{proposition}
For the proof of Proposition~\ref{p1}, see Appendix~\ref{computations}. As a consequence of \eqref{eq:Maxwell1} and Proposition \ref{p1}, the equation \eqref{16'} in vacuum can be rewritten as
\begin{equation}
\begin{array}{ll}\label{eq:Maxwell3}
\begin{cases}
\mep\p_t \hg + \rot \Es=0\, ,\\
\mep\p_t \eg - \rot \Hs=0\, .
\end{cases}
\end{array}
\end{equation}
It is clear that, if the initial data \eqref{18} satisfy \eqref{eq:diveghg} for $t=0$, then the solution of \eqref{eq:Maxwell3} satisfies \eqref{eq:diveghg} for all $t\in [0,T]$. In particular, this yields \eqref{eq:div2}.
Therefore, for such initial data, instead of \eqref{16}--\eqref{18} we can equivalently consider problem \eqref{16}, \eqref{eq:Maxwell3}, \eqref{17}, \eqref{18}, disregarding \eqref{eq:div2} which will be recovered {\it a posteriori}.
Therefore, for initial data as in Proposition \ref{p1}, problem \eqref{4}, \eqref{eq:MaxwellB}, \eqref{eq:15-1}--\eqref{11} can be reformulated on the fixed
reference domains $\Omega^\pm$ as
\begin{equation}
\mathbb{P}(U,\Psi)=0\quad\mbox{in}\ [0,T]\times \Omega^+,\label{16.1}
\end{equation}
\begin{equation}\label{16'.1}
\mep B_0(\Psi)\partial_t\Ws+\sum_{j=1}^3B_j\partial_j\Ws +\mep {B}_4(\Psi)\Ws= 0\quad\mbox{in}\ [0,T]\times \Omega^-,
\end{equation}
\begin{equation}
\mathbb{B}(U,\Wc,\varphi )=0\quad\mbox{on}\ [0,T] \times\Gamma,\label{17.1}
\end{equation}
\begin{equation}
(U,\Wc)|_{t=0}=(U_0,\Wc_0)\quad\mbox{in}\ \Omega^+\times\Omega^-,\qquad \varphi|_{t=0}=\varphi_0\quad \mbox{on}\ \Gamma,\label{18.1}
\end{equation}
where
\[
\mep B_0(\Psi)\partial_t\Ws+\sum_{j=1}^3B_j\partial_j\Ws +\mep {B}_4(\Psi)\Ws= \begin{pmatrix}
\mep \p_t \hg + \rot \Es\\
\mep \p_t \eg - \rot \Hs\\
\end{pmatrix},
\]
\begin{equation}\label{eq:B0}
B_0(\Psi)\doteq  KJ^{-1}>0,\qquad {B}_4(\Psi)=\partial_t{B}_0(\Psi),
\end{equation}
with the matrices~$K,J$ defined in~\eqref{eq:K} and \eqref{eq:J}, respectively. The matrix $B_0(\Psi)$ is also symmetric, as well as the $B_j$, so that the system in \eqref{16'.1} is symmetric hyperbolic.

Due to the particular boundary conditions, problems \eqref{16}--\eqref{18} and \eqref{16.1}--\eqref{18.1} are initial boundary value problems with characteristic boundary. Moreover, since we prescribe a different number of conditions on different portions of the boundary (see \eqref{17}, \eqref{eq:div2}), the boundary is non-uniformly characteristic for \eqref{16}--\eqref{18}, that is characteristic with variable multiplicity, see \cite{nishitani00,secchinonunif1,secchinonunif2}. On the contrary, after the introduction of the new variables $\Ws$ and the new formulation \eqref{16'.1} in the vacuum region, the boundary becomes characteristic  with constant multiplicity  for system \eqref{16.1}--\eqref{18.1} (here we don't assume \eqref{eq:div2}).

{The new formulation \eqref{16.1}--\eqref{18.1} can} be used for the resolution of the problem with variable sign of $\p_t\varphi$ (which yields the change of multiplicity for \eqref{16}--\eqref{18}), that is with both expansion and contraction of the plasma region in vacuum. {This extends the analysis of  \cite{mandrik-trakhinin,trakhinin12}, where only the case of expansion of plasma into vacuum is considered, i.e. it is assumed $\p_t\varphi<0$ throughout the whole boundary.}

For a detailed discussion on the number of boundary conditions in \eqref{16}--\eqref{18} and in \eqref{16.1}--\eqref{18.1} see Appendix \ref{calcolosegni}.

\section{The linearized problem}\label{sec:linearization}


\subsection{Basic state}
\label{s2.2}

Let us denote
\[
Q^\pm_T\doteq   (-\infty,T]\times\Omega^\pm,\quad \omega_T\doteq  (-\infty,T]\times\Gamma.
\]
Let
\begin{equation}
(\hat{U}(t,x ),\hat\Wc (t,x),\hat{\varphi}(t,{x}'))
\label{21}
\end{equation}
be a given sufficiently smooth vector-function
with $\hat{U}=(\hat{q},\hat{v},\hat{H},\hat{S})$ and $\hat\Wc=(\hat\Hc,\hat\Ec)$, respectively defined on $Q^+_T,Q^-_T,\omega_T$, with
\begin{equation}
\begin{array}{ll}\label{}
\|\hat{U}\|_{W^{2,\infty}(Q^+_T)}+
\|\partial_1\hat{U}\|_{W^{2,\infty}(Q^+_T)}+
\|\hat{\Wc}\|_{W^{2,\infty}(Q^-_T)}+
\|\hat{\varphi}\|_{W^{3,\infty}([0,T]\times\R^2)} \leq \kappa,\\
\\
 \| \hat\varphi\|_{{\mathcal C}
([0,T];H^{2}(\R^2))} \le 1,
\label{22}
\end{array}
\end{equation}
where $\kappa>0$ is a constant.
Corresponding to the given $\hat\varphi$, we construct $\hat\Psi$ and the diffeomorphism $\hat\Phi$  as in Lemma~\ref{lemma3} such that
\[
1+\hat{\Psi}_1\geq 1/2.
\]
We assume that the basic state \eqref{21} satisfies the following conditions (which are less restrictive in the vacuum side than in~\cite{trakhinin12}). We have, for some positive $\rho_0,\rho_1\in\R$,
\begin{equation}
\rho (\hat{p},\hat{S})\geq \rho_0 >0,\quad \rho_p(\hat{p},\hat{S})\ge \rho_1 >0 \qquad \mbox{in}\ \overline{Q}^+_T,
\label{23}
\end{equation}
\begin{equation}
\partial_t\hat{H}+\frac{1}{1+\hat\Psi_1}\left\{ (\hat{w} \cdot\nabla )
\hat{H} - (\hat{h} \cdot\nabla ) \hat{v} + \hat{H}{\rm div}\,\hat{u}\right\} =0\qquad \mbox{in}\ Q^+_T,
\label{26}
\end{equation}
\begin{equation}\label{eq:divbasic}
\dv \hat\hg {\,=\dv\hat\eg} =0\qquad \mbox{in}\; Q^-_T,
\end{equation}
\begin{equation}\label{eq:Maxwbasic}
\mep \p_t \hat\hg + \rot \hat\Es = 0 \qquad \mbox{on}\,\; \omega_T,
\end{equation}
\begin{equation}
\partial_t\hat{\varphi}-\hat{v}_{N}=0,\quad \hat\Es_2=\hat\Es_3=0 \qquad \mbox{on}\,\; \omega_T,\label{24}
\end{equation}
%
%
%
where all the ``hat'' functions are determined like for the corresponding values of $(U,\Wc,\varphi)$, i.e.
\[
\hat p=\hat q  -|\hat H  |^2/2 , \qquad \hat{v}_{N}=\hat{v}_1- \hat{v}_2\hat\Psi_2 - \hat{v}_3\hat\Psi_3 ,
\]
\[
\hat{h}=(\hat{H}_{N},\hat{H}_2(1+\hat\Psi_1),\hat{H}_3(1+\hat\Psi_1)), \qquad \hat{\mathfrak{h}}=(\hat{\mathcal{H}}_{N},\hat{\mathcal{H}}_2(1+\hat\Psi_1),\hat{\mathcal{H}}_3(1+\hat\Psi_1)),
\]
\[
\hat{H}_{N}=\hat{{H}}_1- \hat{{H}}_2\hat\Psi_2 - \hat{H}_3\hat\Psi_3 , \qquad \hat{\mathcal{H}}_{N}=\hat{\mathcal{H}}_1- \hat{\mathcal{H}}_2\hat\Psi_2 - \hat{\mathcal{H}}_3\hat\Psi_3 ,
\]
\[
\hat\Es = ( (1+\hat\Psi_1)\hat\Ec_1 , \hat\Ec_2+\hat\Psi_2\hat\Ec_1-\mep \hat\Psi_t\hat\Hc_3 , \hat\Ec_3+\hat\Psi_3\hat\Ec_1+\mep\hat\Psi_t\hat\Hc_2 )\, ,
\]
and where
\[
\hat{u}=(\hat{v}_{N},\hat{v}_2(1+\hat\Psi_1),\hat{v}_3(1+\hat\Psi_1)),\quad
\hat{w}=\hat{u}-(\hat\Psi_t ,0,0).
\]
%
It follows from (\ref{26}) that the constraints
\begin{equation}
\dv \hat{h}=0\quad \mbox{in}\; Q^+_T,\qquad \hat{H}_{N}=0\quad \mbox{on}\; \omega_T
\label{27}
\end{equation}
are satisfied for the basic state (\ref{21}) if they hold at $t=0$ (see \cite{trakhinin09arma} for the proof).
Thus, for the basic state we also require the fulfillment of conditions
\eqref{27} at $t=0$.
Other assumptions on the basic state needed for the stability analysis will be given in the statement of Theorem \ref{main}.

Note that (\ref{22}) yields
\[
 \|\nabla_{t,x}\hat{\Psi}\|_{W^{2,\infty}([0,T]\times\R^3)}\leq C(\kappa),
\]
where $\nabla_{t,x}=(\partial_t, \nabla )$ and $C=C(\kappa)>0$ is a constant depending on $\kappa$.
We also remark that thanks to~\eqref{24} and taking into account the definition of~$\hg, \Es$ (and recalling that~$\hat\Psi_1=0$ on~$\omega_T$), equation~\eqref{eq:Maxwbasic} may be written in the form:
\begin{equation}\label{eq:Maxwbasic2}
\begin{cases}
\p_t \hat \hg_1=0 & \text{on~$\omega_T$,} \\
\mep \p_t \hat \Hc_2 + \p_3 \hat \Ec_1 - \p_1 \hat \Es_3 =0& \text{on~$\omega_T$,} \\
\mep \p_t \hat \Hc_3 + \p_1 \hat \Es_2 - \p_2 \hat \Ec_1 =0& \text{on~$\omega_T$.}
\end{cases}
\end{equation}
\\
It follows from the first equation in~\eqref{eq:Maxwbasic2}, that the constraint
\begin{equation}\label{eq:h1basic}
\hat\hg_1=0\quad \mbox{on}\; \omega_T
\end{equation}%
is satisfied for the basic state (\ref{21}) if it holds at $t=0$. We remark that if we strengthen assumption~\eqref{eq:Maxwbasic} to the following
\begin{equation}\label{eq:Maxbasic}
\mep \p_t \hat{\hg}+\nabla\times\hat \Es=0\qquad \mbox{in}\; \overline{Q}^-_T,
\end{equation}
then also~\eqref{eq:divbasic} follows as a consequence of~\eqref{eq:Maxbasic}, provided that 
it holds for~$t=0$.
\begin{notation}\label{not:hat}
From now on, we denote by $\Hs, \Es, \hg, \eg$ variables defined using the basic state $\hat\Psi$ (instead of $\Psi$) and $\Wc$ (without ``hat''), while ``hat''-variables $\hat\Hs, \hat\Es, \hat\hg, \hat\eg$ are defined using the basic state for all terms ($\hat\Psi$ and $\hat\Wc$). For instance,
\begin{align*}
\Hs_1 = (1+\hat\Psi_1)\Hc_1\, , \qquad \hat\Hs_1 = (1+\hat\Psi_1)\hat\Hc_1\, .
\end{align*}
\end{notation}


\subsection{Linearized problem}
\label{s2.3}

{We want to linearize \eqref{16.1}--\eqref{17.1} (or \eqref{16}--\eqref{17}), and in particular \eqref{16'.1} in vacuum, since in this formulation we have the main advantage of a characteristic boundary of constant multiplicity. However, the linearization can be performed more easily, by resorting to standard techniques, if we recast \eqref{16'.1} in terms of $\Wc$ (clearly, the multiplicity remains the same). We recall that $K$ has positive eigenvalues, hence it is invertible and the multiplication by $K$ does not alter the number of incoming characteristics of a system, so that we may consider \eqref{16'.1} multiplied on the left by $K^{-1}$. Noticing that  \eqref{eq:Maxwell1} multiplied by $K^{-1}$ is indeed $\tilde{\mathbb V}(\Wc, \Psi)$, we obtain
\begin{align} \label{eq.V}
\mathbb{V}(\Wc, \Psi) := \tilde{\mathbb V}(\Wc, \Psi) - \mep K^{-1} \begin{pmatrix} Q & \underline{0}^T \\ \underline{0}^T & Q \end{pmatrix} \begin{pmatrix} \dv (\eta\Hc) \\ \dv(\eta\Ec)\end{pmatrix} = 0\, .
\end{align}}

{The linearized equations for \eqref{16.1}, \eqref{eq.V}, \eqref{17.1} read:
\[
\mathbb{P}'(\hat{U},\hat{\Psi})(\delta U,\delta\Psi)\doteq
\frac{\rm d}{{\rm d}\theta}\mathbb{P}(U_{\theta},\Psi_{\theta})|_{\theta =0}=f
\qquad \mbox{in}\ Q^+_T,
\]
\[
\mathbb{V}'(\hat\Wc,\hat{\Psi})(\delta\Wc,\delta\Psi)\doteq
\frac{\rm d}{{\rm d}\theta}\mathbb{V}(\Wc_\theta,\Psi_{\theta})|_{\theta =0}=\chi
\qquad \mbox{in}\ Q^-_T,
\]
\[
\mathbb{B}'(\hat{U},\hat\Wc,\hat{\varphi})(\delta U,\delta\Wc,\delta \varphi )\doteq
\frac{\rm d}{{\rm d}\theta}\mathbb{B}(U_{\theta},\Wc_\theta,\varphi_{\theta})|_{\theta =0}={g}
\qquad \mbox{on}\ \omega_T,
\]
where $U_{\theta}=\hat{U}+ \theta\,\delta U$, $\Wc_{\theta}=
\hat{\Wc}+\theta\,\delta \Wc$,
$\varphi_{\theta}=\hat{\varphi}+ \theta\,\delta \varphi$;
$\delta\Psi$ is constructed from $\delta \varphi$ as in Lemma \ref{lemma3} and
$\Psi_{\theta}=\hat\Psi +  \theta\,\delta\Psi$.}

Here we introduce the source terms $f=(f_1,\ldots ,f_8)$, $\chi=(\chi_1, \ldots
\chi_6)$ and $g=(g_1,g_2,g_3,g_4)$ to make the interior equations and the boundary conditions inhomogeneous.

\subsubsection{Vacuum part.} First, we compute the exact form of the linearized equations in $Q^-_T$ (below we drop $\delta$). {By exploiting \eqref{eq:divbasic}, it is standard to obtain (see~\cite{SeTr})
\begin{align*}
\frac{\rm d}{{\rm d}\theta}\mathbb{V}(\Wc_\theta,\Psi_{\theta})|_{\theta =0}
= \mathbb{V}( \Wc,\hat{\Psi}) - \, \left\{ \mathbb{V}(\hat{\Psi}) \Psi \right\} \frac{\p_1 \hat{\Wc}}{1+\hat{\Psi}_1} \, .
\end{align*}
Note that~$\mathbb{V}(\hat{\Psi})\Psi$ denotes the matrix obtained from $\mathbb V(\Wc, \Psi)$ with $\hat\Psi$ instead of $\Psi$ and where all derivatives are applied to $\Psi$.}

Since $\mathbb{V}'(\hat{\Wc},\hat{\Psi})({\Wc},{\Psi})$ is a first-order differential operator in $\Psi$,
as in  \cite{alinhac}, the linearized problem is rewritten in terms of the ``good unknown''
\begin{equation}
\dot{\Wc}\doteq  \Wc -\frac{\Psi}{1+\hat\Psi_1}\,\partial_1\hat{\Wc}\, ,
\label{29}
\end{equation}
so that, {again by standard computations, we deduce}
%
{\begin{align*}
\mathbb{V}(\dot{\Wc},\hat{\Psi}) + \frac\Psi{1+\hat{\Psi}_1}\, \p_1 \left\{ \mathbb{V}(\hat{\Wc},\hat{\Psi})  \right\} = \chi\, .
\end{align*}}

We can consider
\[ 
\frac\Psi{1+\hat{\Psi}_1}\, \p_1 \left\{ \mathbb{V}{( \hat{\Wc},\hat{\Psi})} \right\} \]
as an error term for the nonlinear analysis that we will address in a future work, so that we get the system
\[ \mathbb{V}{(\dot \Wc,\hat{\Psi})} = \chi \, ,\]
which has the same form of the starting system, but with coefficients depending on $\hat{\Psi}$ instead of~$\Psi$.

\subsubsection{Plasma part.} Proceeding similarly as in the vacuum part, we linearize, introduce the good unknown and remove error terms to obtain (see~\cite{cs, SeTr})
\[
P(\hat{U},\hat{\Psi}){\dot U} +{\mathcal C}(\hat{U},\hat{\Psi})
{\dot U}=f\, ,
\]
where the matrix ${\mathcal C}(\hat{U},\hat{\Psi})$ is determined as follows:
\[
\begin{array}{r}
{\mathcal C}(\hat{U},\hat{\Psi})Y
= (Y ,\nabla_yA_0(\hat{U} ))\partial_t\hat{U}
 +(Y ,\nabla_y\tilde{A}_1(\hat{U},\hat{\Psi}))\partial_1\hat{U}
 \\[6pt]
+ (Y ,\nabla_yA_2(\hat{U} ))\partial_2\hat{U}
+ (Y ,\nabla_yA_3(\hat{U} ))\partial_3\hat{U},
\end{array}
\]
\[
(Y ,\nabla_y A(\hat{U}))\doteq  \sum_{i=1}^8y_i\left.\left(\frac{\partial A (Y )}{
\partial y_i}\right|_{Y =\hat{U}}\right),\quad Y =(y_1,\ldots ,y_8).
\]

\subsubsection{Boundary conditions.} We have

\[ \mathbb{B}'(\hat{U},\hat{\Wc},\hat{\Psi})(U,\Wc,\varphi) = \begin{pmatrix}
\varphi_t+\hat{v}_2\varphi_2+\hat{v}_3\varphi_3-v_N \\
q-\hat{\Hc}\cdot\Hc +\hat{\Ec}\cdot\Ec\\
\Es_2 -\mep\varphi_t\hat{\Hc}_3+\varphi_2\hat{\Ec}_1\\
\Es_3 + \mep\varphi_t\hat{\Hc}_2+\varphi_3\hat{\Ec}_1
\end{pmatrix}_{|_{x_1=0}} \, , \]
where $q\doteq  p+ \hat{H}\cdot H$ and $v_{N}\doteq   v_1-v_2\hat\Psi_2 -v_3\hat\Psi_3 $.
Taking into account assumptions \eqref{24} and recalling that $\Psi\equiv\varphi$ on $\omega_T$,
we rewrite our linearized equations in terms of the good unknowns:
\begin{align}\begin{split}
\mathbb{B}'(\hat{U},\hat{\Wc},\hat{\varphi})(\dot{U},\dot{\Wc},\varphi )& \doteq   \mathbb{B}'(\hat{U},\hat{\Wc},\hat{\varphi})(U,\Wc,\varphi )\\[6pt] &  = \begin{pmatrix}
\varphi_t+\hat{v}_2\varphi_2+\hat{v}_3\varphi_3-\dot v_N -\varphi\p_1\hat v_N\\
\dot q-\hat{\Hc}\cdot\dot\Hc +\hat{\Ec}\cdot\dot\Ec + \varphi[\p_1\hat q]\\
\dot \Es_2 -\mep\varphi_t\hat{\Hc}_3+\varphi_2\hat{\Ec}_1 +\varphi\p_1\hat\Es_2\\
\dot\Es_3 + \mep\varphi_t\hat{\Hc}_2+\varphi_3\hat{\Ec}_1 +\varphi\p_1\hat\Es_3
\end{pmatrix}_{|_{x_1=0}} \, ,
\label{32}\end{split}
\end{align}
where $\dot{v}_{N}=\dot{v}_1-\dot{v}_2\hat\Psi_2 -\dot{v}_3\hat\Psi_3 $,
$\dot{\mathcal{H}}_{N}=\dot{\mathcal{H}}_1-\dot{\mathcal{H}}_2\hat\Psi_2 -\dot{\mathcal{H}}_3\hat\Psi_3 $, and
\[
[\partial_1\hat{q}]=(\partial_1\hat{q})|_{x_1=0}-(\hat{\mathcal{H}} \cdot \partial_1\hat{\mathcal{H}})|_{x_1=0}
+(\hat{\mathcal{E}} \cdot \partial_1\hat{\mathcal{E}})|_{x_1=0}\,.
\]
Thanks to the second and third equations in~\eqref{eq:Maxwbasic2}, it follows that
\begin{equation}\label{eq:newboundaryc}
\begin{pmatrix}
-\mep\varphi_t\hat{\Hc}_3+\varphi_2\hat{\Ec}_1 +\varphi\p_1\hat\Es_2\\
\mep\varphi_t\hat{\Hc}_2+\varphi_3\hat{\Ec}_1 +\varphi\p_1\hat\Es_3
\end{pmatrix}_{|_{x_1=0}} =
\begin{pmatrix}
-\mep\p_t(\varphi\hat{\Hc}_3)+\p_2(\varphi\hat{\Ec}_1) \\
\mep\p_t(\varphi\hat{\Hc}_2)+\p_3(\varphi\hat{\Ec}_1)
\end{pmatrix}_{|_{x_1=0}}\,,
\end{equation}
%
%
hence, we obtain
\begin{align}\begin{split}
\begin{pmatrix}
\varphi_t+\hat{v}_2\varphi_2+\hat{v}_3\varphi_3- \dot v_N -\varphi\p_1\hat v_N\\
\dot q-\hat{\Hc}\cdot\dot\Hc +\hat{\Ec}\cdot\dot\Ec + \varphi[\p_1\hat q]\\
\dot \Es_2 -\mep\p_t(\varphi\hat{\Hc}_3)+\p_2(\varphi\hat{\Ec}_1)\\
\dot \Es_3 + \mep\p_t(\varphi\hat{\Hc}_2)+\p_3(\varphi\hat{\Ec}_1)
\end{pmatrix}_{|_{x_1=0}} = g \, .
\label{32bis}\end{split}
\end{align}

\subsubsection{Conclusion.} The new form of our linearized problem for $(\dot{U},\dot{\mathcal{W}},\varphi )$ reads:
\begin{subequations}\label{34}
\begin{align}
\hat{A}_0\partial_t {\dot U}+\sum_{j=1}^{3}\hat{A}_j\partial_j {\dot U}+
\hat{\mathcal C} {\dot U}=f \qquad &\mbox{in}\ Q^+_T,\label{34a}
\\
\begin{pmatrix} \label{36}
\mep\p_t \dot\hg + \rot \dot\Es\\
\mep\p_t \dot\eg - \rot \dot\Hs
\end{pmatrix} =\chi \qquad &\mbox{in}\ Q^-_T,
\\
\varphi_t=\dot{v}_{N}-\hat{v}_2 \varphi_2-\hat{v}_3\varphi_3 + \varphi\,\partial_1\hat{v}_{N}+g_1,  \qquad &\label{34b}
\\
\dot{q}=\hat{\mathcal{H}}\cdot\dot{\mathcal{H}}- \hat\Ec\cdot\dot\Ec -  [ \partial_1\hat{q}] \varphi +g_2, \qquad & \label{35}
\\
\dot \Es_2 = \mep\p_t(\varphi\hat{\Hc}_3) - \p_2(\varphi\hat{\Ec}_1) + g_3 \qquad &\mbox{on}\ \omega_T, \label{38}\\
\dot\Es_3 = -\mep\p_t(\varphi\hat{\Hc}_2)-\p_3(\varphi\hat{\Ec}_1)+g_4\qquad &\mbox{on}\ \omega_T,
\label{37}
\\
(\dot{U},\dot{\Wc},\varphi )=0\qquad &\mbox{for}\ t<0,\label{38a}
\end{align}
\end{subequations}
%
where
\[
\hat{A}_{\alpha}\doteq  {A}_{\alpha}(\hat{U}),\quad \alpha =0,2,3,\quad
\hat{A}_1\doteq \tilde{A}_1(\hat{U},\hat{\Psi}),\quad
\hat{\mathcal C}\doteq  {\mathcal C}(\hat{U},\hat{\Psi}),
\]
and the ``dot''-variables $\dot\Hs, \dot\Es, \dot\hg, \dot\eg$ are defined analogously to $\Hs, \Es, \hg, \eg$, using $\hat\Psi$ and $\dot\Wc$ instead of $\Wc$.

We assume that the source term $ \chi$ of~\eqref{36} satisfies the constraint
\begin{equation}\label{eq:constraintchi}
\dv \Xi'=\dv \Xi''=0\,, \qquad \text{where} \quad \Xi'=(\chi_1,\chi_2,\chi_3)\,, \quad \Xi''=(\chi_4,\chi_5,\chi_6)\,,
\end{equation}
that  the source terms $f, \chi$ and the boundary datum $g$ vanish in the past, and we consider the case of zero initial data. We postpone the case of nonzero initial data to the nonlinear analysis of a future work (see e.g. \cite{cs,SeTrNl,trakhinin09arma}).

\subsection{Reduction to homogeneous constraints in the \lq\lq vacuum part\rq\rq}

We decompose $\dot{\Wc}$ in \eqref{34} as $\dot{\Wc}=\Wc'+\Wc''$ (and accordingly $\dot{\Hs}=\Hs'+\Hs''$, and similarly for $\dot\Es, \dot\hg, \dot\eg$), where $\Wc''$ is required to solve for each $t$
\begin{equation}
\begin{array}{rl}\label{elliptic}
\begin{pmatrix}
\mep\p_t \hg'' + \rot \Es''\\
\mep\p_t \eg'' - \rot \Hs''
\end{pmatrix}=\chi \qquad &\mbox{in}\ \Omega^-,
\\*[1em]
\Es''_2 = g_3, \quad \Es''_3=g_4\qquad &\mbox{on}\ \Gamma.
\end{array}
\end{equation}
The source term $ \chi$ of the first equation should satisfy the constraint~\eqref{eq:constraintchi}.
By classical results on Maxwell's equations, we have the following.
\begin{lemma}\label{elliptic2}
Assume that the data $\chi,g_3,g_4$ in \eqref{elliptic}, vanishing in appropriate way as $x$ goes to infinity, satisfy the constraints \eqref{eq:constraintchi}. Then there exists a solution $\Wc''$ of \eqref{elliptic} vanishing at infinity.
\end{lemma}
In the statement of the lemma above we intentionally leave unspecified the description of the regularity and the behavior at infinity of the data and consequently of the solution. This point will be faced in the forthcoming paper on the resolution of the nonlinear problem.

Given $\Wc''$, now we look for $\Wc'$ such that
\begin{equation}
\begin{array}{ll}\label{postelliptic}
\mep\p_t \hg' + \rot \Es' =0,\qquad \mep\p_t \eg' - \rot \Hs'=0 \qquad &\mbox{in}\ Q^-_T,
\\
\dot{q}=\hat{\mathcal{H}}\cdot{\mathcal{H}}' - \hat\Ec\cdot\Ec'-  [ \partial_1\hat{q}] \varphi +g_2' , \qquad &
\\
\Es_2' = \mep\p_t(\varphi\hat{\Hc}_3) - \p_2(\varphi\hat{\Ec}_1) , \qquad  \qquad &
\\
\Es_3' = -\mep\p_t(\varphi\hat{\Hc}_2)-\p_3(\varphi\hat{\Ec}_1) &\mbox{on}\ \omega_T,
\end{array}
\end{equation}
where we have denoted $g_2'=g_2+\hat{\mathcal{H}}\cdot{\mathcal{H}}''-\hat\Ec\cdot\Ec''$. If $\Wc''$ solves \eqref{elliptic} and $\Wc'$ is a solution of \eqref{postelliptic}, then $\dot{\Wc}=\Wc'+\Wc''$ clearly solves \eqref{36}, \eqref{35}{--}\eqref{37}.

From \eqref{34}, \eqref{postelliptic}, the new form of the reduced linearized problem with unknowns ($\dot U,\Wc'$),  dropping for convenience the prime sign in $\Wc',g_2'$ and the dot sign in $\dot U$, reads
\begin{subequations}\label{34''}
\begin{align}
\hat{A}_0\partial_t{U}+\sum_{j=1}^{3}\hat{A}_j\partial_j{U}+
\hat{\mathcal C}{U}=f \qquad &\mbox{in}\ Q^+_T,\label{34''a}
\\
\begin{pmatrix}
\mep\p_t \hg + \rot \Es\\
\mep\p_t \eg - \rot \Hs
\end{pmatrix} = {0} \qquad &\mbox{in}\ Q^-_T, \label{36''b}
\\
\varphi_t={v}_{N}-\hat{v}_2 \varphi_2-\hat{v}_3\varphi_3 + \varphi\,\partial_1\hat{v}_{N} +g_1 , \qquad &\label{37c}
\\
 {q}=\hat{\mathcal{H}}\cdot{\mathcal{H}}-\hat\Ec\cdot\Ec -  [ \partial_1\hat{q}] \varphi  +g_2 , \qquad & \label{37d}
\\
\Es_2 =\mep\p_t(\varphi\hat{\Hc}_3) - \p_2(\varphi\hat{\Ec}_1) , \qquad & \label{37''}
\\
\Es_3 = -\mep\p_t(\varphi\hat{\Hc}_2)-\p_3(\varphi\hat{\Ec}_1) \qquad &\mbox{on}\ \omega_T, \label{37''bis}
\\
({U},{\Wc},\varphi )=0\qquad & \mbox{for}\ t<0.\label{38''f}
\end{align}
\end{subequations}
%
\subsection{Reduction to homogeneous constraints in the \lq\lq plasma part\rq\rq}
From problem \eqref{34''} we can deduce nonhomogeneous equations associated with the divergence constraint ${\rm div}\,{h}=0$ and the ``redundant'' boundary conditions ${H_N}|_{x_1=0}=0$ for the nonlinear problem. Proceeding as in \cite{SeTr}, Proposition 7, we can
reduce \eqref{34''} to a problem with homogeneous constraints \eqref{93} and \eqref{95} in terms of a new variable $U^\natural$.

Dropping for convenience the indices $^{\natural}$, the new form of our reduced linearized problem now reads
\begin{subequations}\label{34'}
\begin{align}
\hat{A}_0\partial_t{U}+\sum_{j=1}^{3}\hat{A}_j\partial_j{U}+
\hat{\mathcal C}{U}=F \qquad &\mbox{in}\ Q^+_T,\label{34'a}
\\
\begin{pmatrix}
\mep\p_t \hg + \rot \Es\\
\mep\p_t \eg - \rot \Hs\\
\end{pmatrix} = {0} \qquad &\mbox{in}\ Q^-_T, \label{36'b}
\\
\varphi_t={v}_{N}-\hat{v}_2 \varphi_2-\hat{v}_3\varphi_3 + \varphi\,\partial_1\hat{v}_{N} \qquad &
\\
{q}=\hat{\mathcal{H}}\cdot{\mathcal{H}}-\hat\Ec\cdot\Ec -  [ \partial_1\hat{q}] \varphi \qquad & \label{35'd}
\\
\Es_2 =\mep\p_t(\varphi\hat{\Hc}_3) - \p_2(\varphi\hat{\Ec}_1) \qquad &
\label{37'}
\\
\Es_3 = -\mep\p_t(\varphi\hat{\Hc}_2)-\p_3(\varphi\hat{\Ec}_1) \qquad &\mbox{on}\ \omega_T, \label{37'bis}
\\
({U},{\Wc},\varphi )=0\qquad & \mbox{for}\ t<0 ,\label{38'f}
\end{align}
\end{subequations}
and solutions should satisfy
\begin{equation}
{\rm div}\,{h}=0\qquad\mbox{in}\ Q^+_T,
\label{93}
\end{equation}
\begin{equation}
{H}_{N}=\hat{H}_2\varphi_2 +\hat{H}_3\varphi_3 -
\varphi\,\partial_1\hat{H}_{N}\quad\mbox{on}\ \omega_T.
\label{95}
\end{equation}
All the notations here for $U$ and $\Wc$ (e.g., $h$, $\Hs$, $\hg$, etc.) are analogous to the corresponding ones for $\dot{U}$ and $\dot{\Wc}$ introduced above.
		
\subsection{An equivalent formulation of \eqref{34'} }\label{equiva}
In the following analysis it is convenient to make use of different ``plasma'' variables and an equivalent form of equations \eqref{34'a}.
With the usual notation, we define the matrix
\begin{equation*}
\begin{array}{ll}
\hat \eta=\begin{pmatrix}
 1&-\hat\Psi_2 &-\hat\Psi_3 \\
0 &1+\hat\Psi_1&0\\
0&0&1+\hat\Psi_1
\end{pmatrix}.
\end{array}
\end{equation*}
It follows that
\begin{equation}
\begin{array}{ll}\label{defcalU}
{u}=({v}_{N},{v}_2(1+\hat\Psi_1),{v}_3(1+\hat\Psi_1))=\hat \eta\, v, \\
{h}=({H}_{N},{H}_2(1+\hat\Psi_1),{H}_3(1+\hat\Psi_1))=\hat\eta \,H.
\end{array}
\end{equation}
Multiplying \eqref{34'a} on the left side by the matrix
\begin{equation*}
\begin{array}{ll}
\hat R=\begin{pmatrix}
 1&\underline 0&\underline 0&0 \\
 \underline0^T&\hat \eta&0_3& \underline 0^T\\
 \underline 0^T&0_3&\hat \eta&\underline 0^T\\
 0&\underline 0^T&\underline 0^T&1
\end{pmatrix},
\end{array}
\end{equation*}
after some calculations we get the symmetric hyperbolic system for the new vector of unknowns $\mathcal{U}=(q,u,h,S)$
(compare with \eqref{3'}, \eqref{34'a}):
\begin{equation}\begin{split}
\label{34'''}
&(1+\hat\Psi_1)
\left(\begin{matrix}
{\hat\rho_p/\hat\rho}&\underline
0&-({\hat\rho_p/\hat\rho})\hat h &0 \\
\underline 0^T&\hat\rho
\hat a_0&0_3&\underline 0^T\\
-({\hat\rho_p/\hat\rho})\hat h^T&0_3&\hat a_0 +({\hat\rho_p/\hat\rho})\hat h\otimes\hat h&\underline 0^T\\
0&\underline 0&\underline 0&1
\end{matrix}\right)\dt
\left(\begin{matrix}
q \\ u \\
h\\S \end{matrix}\right)\\
\\
& \qquad +
\left( \begin{matrix}
0&\nabla\cdot&\underline 0&0\\
\nabla&0_3&0_3 &\underline 0^T\\
\underline 0^T&0_3 &0_3&\underline 0^T\\
0&\underline 0&\underline 0&0
\end{matrix}\right)
\left(\begin{matrix}q \\ u \\ h\\S \end{matrix}\right)
\\
 \\
& +
(1+\hat\Psi_1)
\left( \begin{matrix}
(\hat\rho_p/\hat\rho)
\hat w \cdot\nabla&\nabla\cdot&-({\hat\rho_p/\hat\rho})\hat h\hat w \cdot\nabla&0\\
\nabla&\hat\rho \hat a_0\hat w \cdot\nabla&-\hat a_0\hat h \cdot\nabla &\underline 0^T\\
-({\hat\rho_p/\hat\rho})\hat h^T \hat w \cdot\nabla&-\hat a_0\hat h \cdot\nabla &(\hat a_0 +({\hat\rho_p/\hat\rho})\hat h\otimes\hat h)
\hat w \cdot\nabla&\underline 0^T\\
0&\underline 0&\underline 0&\hat w \cdot\nabla
\end{matrix}\right)
\left(\begin{matrix}q \\ u \\ h\\S \end{matrix}\right)\\ \\
&\qquad +\hat{\mathcal{C}}'\mathcal{U}=\mathcal{F}\,,
\end{split}
\end{equation}
where $\hat a_0$ is the symmetric and positive definite matrix
$$
\hat a_0 =(\hat \eta^{-1})^T\hat \eta^{-1},$$
$\hat{\mathcal{C}}'$ is a new zero-order term (a matrix whose precise form has no importance) and where we have set
$
\mathcal{F}=(1+\hat\Psi_1)  \hat R
F.$
We write system \eqref{34'''} in compact form as
\begin{equation}
\begin{array}{ll}
\displaystyle \hat{\mathcal{A}}_0\partial_t{\mathcal{U}}+\sum_{j=1}^{3}(\hat{\mathcal{A}}_j+E_{1j+1})\partial_j{\mathcal{U}}+
\hat{\mathcal C}'{\mathcal{U}}=\mathcal{F} ,\label{73}
\end{array}
\end{equation}
where $E_{1j+1}$ denotes the $8\times8$ matrix with $0$ entries, exception given for the $(1,j+1)$ and $(j+1,1)$ entries, which assume value~$1$.
\\
The formulation \eqref{73} has the advantage of the form of the boundary matrix of the system $\hat{\mathcal{A}}_1+{{E}}_{12}$, where
\begin{equation}
\label{a10}
\hat{\mathcal{A}}_1=0 \qquad\mbox{on }\omega_T,
\end{equation}
because $\hat w_1=\hat h_1=0$, and ${{E}}_{12}$ is a constant matrix.
Thus system \eqref{73} is symmetric hyperbolic with characteristic boundary of constant multiplicity (see \cite{rauch85,secchi95,secchi96} for maximally dissipative boundary conditions). Thus, the final form of our reduced linearized problem is
\begin{subequations}\label{34'new}
\begin{align}
\hat{\mathcal A}_0\partial_t{\mathcal U}+\sum_{j=1}^{3}(\hat{\mathcal A}_j+ E_{1 j+1})\partial_j{\mathcal U}+
\hat{\mathcal C'}{\mathcal U}=\mathcal F \qquad &\mbox{in}\ Q^+_T,\label{34'anew}
\\
\begin{pmatrix}
\mep\p_t \hg + \rot \Es\\
\mep\p_t \eg - \rot \Hs
\end{pmatrix} = {0} \qquad &\mbox{in}\ Q^-_T, \label{36'bnew}
\\
\varphi_t={v}_{N}-\hat{v}_2 \varphi_2-\hat{v}_3\varphi_3 + \varphi\,\partial_1\hat{v}_{N} , \qquad &\label{37'cnew}
\\
{q}=\hat{\mathcal{H}}\cdot{\mathcal{H}}-\hat\Ec\cdot\Ec -  [ \partial_1\hat{q}] \varphi , \qquad & \label{37'dnew}
\\
\Es_2 =\mep\p_t(\varphi\hat{\Hc}_3) - \p_2(\varphi\hat{\Ec}_1) , \qquad & \label{37'new}
\\
\Es_3 = -\mep\p_t(\varphi\hat{\Hc}_2)-\p_3(\varphi\hat{\Ec}_1) \qquad &\mbox{on}\ \omega_T, \label{37'bisnew}
\\
({\mathcal U},{\Wc},\varphi )=0\qquad & \mbox{for}\ t<0.\label{38'fnew}
\end{align}
\end{subequations}

The solutions $(\mathcal{U}, \Wc)$ to problem \eqref{34'new} satisfy
\begin{eqnarray}
{\rm div}\,{h}=0\qquad&\mbox{in}\ Q^+_T,
\label{84}
\\
{\rm div}\,{\mathfrak{h}}=0,\quad {\rm div}\,{\mathfrak{e}}=0\qquad&\mbox{in}\ Q^-_T,
\label{85}
\\
{h}_{1}=\hat{H}_2\varphi_2 +\hat{H}_3\varphi_3 -
\varphi\partial_1\hat{H}_{N},\qquad&
\label{86}
\\
{\mathfrak{h}}_{1} =\partial_2\bigl(\hat{\mathcal{H}}_2\varphi \bigr) +\partial_3\bigl(\hat{\mathcal{H}}_3\varphi \bigr)\qquad &\mbox{on}\ \omega_T\, ,
\label{87}
\end{eqnarray}
because \eqref{84}--\eqref{87} are just restrictions on the initial data which are automatically satisfied in view of \eqref{38'fnew}.
For instance, equations \eqref{85} trivially follow from \eqref{36'bnew} and \eqref{38'fnew}. Moreover, condition \eqref{87} is obtained by considering equation $\mep\p_t\hg_1+\p_2\Es_3-\p_3\Es_2=0$ in \eqref{36'bnew} at $x_1=0$ and taking into account \eqref{37'new}, \eqref{37'bisnew}.


System~\eqref{36'bnew} can be written in terms of $\Ws=(\Hs,\Es)\tm$ if~$|\hat \Psi_t|<{\mep^{-1}}$, i.e. the matrix~$J$ in~\eqref{eq:J} is invertible. Similarly to Hypothesis~\ref{hyp:invertibility0}, we have the following assumption involving the plasma normal speed at the boundary of the basic state.
\begin{hypothesis}\label{hyp:invertibility}
{We assume that}~$\hat v_N$ is sufficiently small {with respect to~$\mep^{-1}$} on~{$\omega_T$}, namely
\begin{equation}\label{eq:smallvn}
\frac1{\sqrt{2\pi}}\,\|\hat v_N(t,0,\cdot)\|_{H^{\frac32}(\R^2)}< {\mep^{-1}}\,.
\end{equation}
\end{hypothesis}
By virtue of~\eqref{eq:Phitcontrol} in correspondence of~$\hat\Psi$ and~$\hat\varphi$, {from Hypothesis~\ref{hyp:invertibility}} it follows that $\sup |\mep\hat\Psi_t(t,x)|<1$ in $[0,T]\times\R^3$, and~$J$ (where we clearly replace~$\Psi$ by~$\hat\Psi$) becomes invertible. We also define~$J_1$ as in~\eqref{eq:J1}, but replacing~$\Psi$ by~$\hat\Psi$.

\medskip

{Since} the basic state satisfies~\eqref{eq:smallvn}, we could write~\eqref{36'bnew} as a symmetric hyperbolic system:
\begin{equation}\label{newmaxwell}
\mep B_0\partial_t\Ws+\sum_{j=1}^3B_j\partial_j\Ws +\mep {B}_4\Ws=0,
\end{equation}
where
$B_0,B_4$ are as in~\eqref{eq:B0}, replacing $\eta$ by~$\hat\eta$, as usual. 
Let us prove that~$B_0>0$. If we define
\begin{equation}\label{eq:hatN}
\hat N = (1, -\hat\Psi_2,-\hat\Psi_3)\,,
\end{equation}
then the characteristic polynomial of~$B_0$ is given by the square of
\begin{align*} &\tau^3 - (|\hat N|^2 +2(\hat\Psi_1+1)^2 - {\mep^2}\hat\Psi_t^2) \tau^2 \\
&\qquad + (\hat\Psi_1+1)^{2}\,(1-{\mep^2}\hat\Psi_t^2)\,(|\hat N|^2+(1+\hat\Psi_1)^2+1) \tau  - (1+\hat\Psi_1)^4\,(1-{\mep^2}\hat\Psi_t^2)^2\,, \end{align*}
hence all the eigenvalues are positive, by virtue of Descartes' sign rule, and using~$\mep |\hat\Psi_t|<1\leq |\hat N|$.

We observe that (recalling that $\hat\Psi_1=0$ on $\omega_T$, in particular $\Hs_1=\Hc_1$ and $\Es_1=\Ec_1$) on $\omega_T$:
\begin{equation}
\begin{array}{ll}\label{nuovebc}
 \hat{\Hc}\cdot\Hc = \hat\hg_1\Hc_1+\hat\Hc_2(\Hc_2+\hat\varphi_2\Hc_1)+\hat\Hc_3(\Hc_3+\hat\varphi_3\Hc_1)\\
\qquad =\hat\hg\cdot\Hs +\mep\hat\varphi_t(-\hat\Hc_2\Ec_3+\hat\Hc_3\Ec_2),\\
\\
\hat\Ec\cdot \Ec
= \hat\Ec_1\eg_1+(\hat\Ec_2+\hat\varphi_2\hat\Ec_1)\Ec_2+(\hat\Ec_3+\hat\varphi_3\hat\Ec_1)\Ec_3
=\hat\Es\cdot\eg +\mep\hat\varphi_t(\hat\Hc_3\Ec_2-\hat\Hc_2\Ec_3).
\end{array}
\end{equation}
Thus we may replace in~\eqref{37'dnew}:
\begin{equation}\label{eq:hHEe}
\hat{\mathcal{H}}\cdot{\mathcal{H}}-\hat\Ec\cdot\Ec = \hat\hg\cdot\Hs-\hat\Es\cdot\eg \equiv \hat\hg_2 \Hs_2+\hat\hg_3\Hs_3-\hat\Es_1\eg_1\,,
\end{equation}
since $\hat\hg_1=\hat\Es_2=\hat\Es_3=0$ on $\omega_T$.



\begin{remark}
The invertible part of the boundary matrix of a system allows to control the trace at the boundary of the so-called noncharacteristic component of the vector solution. Thus, with system \eqref{34'anew} (whose boundary matrix is $-E_{12}$, because of \eqref{a10}), we have the control of $q,u_1$ at the boundary; therefore the components of ${\mathcal{U}}$ appearing in the boundary conditions {\eqref{37'cnew}, \eqref{37'dnew}} are well defined.
\\
The same holds true for \eqref{36'bnew} where we can get the control of $\Hs_2,\Hs_3, \Es_2,\Es_3$, in particular. The control of $\eg_1$ (which appears in \eqref{eq:hHEe}) is not given by the system \eqref{36'bnew}, but by the constraint \eqref{85}. 
\end{remark}

Before studying problem \eqref{34'new}, we should be sure that the number of boundary conditions is in agreement with the number of incoming characteristics for the hyperbolic systems in \eqref{34'new}. Since one of the four boundary conditions \eqref{37'cnew}--\eqref{38'fnew} is needed for determining the function $\varphi (t,x')$, the total number of ``incoming'' characteristics should be three. Let us check that this is true.
\begin{proposition}\label{nrbc}
System \eqref{34'anew} has one incoming characteristic for the boundary $\omega_T$ of the domain $Q_T^+$. System \eqref{36'bnew} has two incoming characteristics for the boundary $\omega_T$ of the domain $Q_T^-$.
\end{proposition}

\begin{proof}
Consider first system \eqref{34'anew}. In view of~\eqref{a10}, the boundary matrix on $\omega_T$ is $-{E}_{12}$, which has one negative (incoming in the domain $Q_T^+$) and one positive eigenvalue, while all other eigenvalues are zero.
\\
Now consider system \eqref{36'bnew}. The boundary matrix $B_1$ (see~\eqref{eq:Bmatrix}) has eigenvalues~$\lambda_\pm=\pm1$, $\lambda_0=0$, each one with multiplicity~$2$. Thus, system \eqref{36'bnew} has indeed two incoming characteristics in the domain $Q_T^-$. 
\end{proof}




\section{Function Spaces}\label{fs}
Now we introduce the main function spaces to be used in the following.
Let us denote
\begin{equation}
\begin{array}{ll}\label{defQ}
Q^\pm\doteq  \R_t\times\Omega^\pm,\quad \omega\doteq \R_t\times\Gamma.
\end{array}
\end{equation}

\subsection{Weighted Sobolev spaces}
For $\gamma\ge 1$ and $s\in\mathbb{R}$, we set
\begin{equation*}\label{weightfcts}
\lambda^{s,\gamma}(\xi)\doteq (\gamma^2+|\xi|^2)^{s/2}
\end{equation*}
and, in particular, $\lambda^{s}\doteq \lambda^{s,1}$.
\newline
Throughout the paper, for real $\gamma\ge 1$ and $n\ge2$, $H^s_{\gamma}(\mathbb{R}^n)$ will denote the Sobolev space of order $s$, equipped with the $\gamma-$depending norm $||\cdot||_{s,\gamma}$ defined by
\begin{equation}\label{normagamma}
||u||^2_{s,\gamma}\doteq (2\pi)^{-n}\int_{\mathbb{R}^n}\lambda^{2s,\gamma}(\xi)|\hat{u}(\xi)|^2d\xi\,,
\end{equation}
$\hat{u}$ being the Fourier transform of $u$. The norms defined by \eqref{normagamma}, with different values of the parameter $\gamma$, are equivalent each other. For $\gamma=1$ we set for brevity $||\cdot||_{s}\doteq ||\cdot||_{s,1}$ (and, accordingly, the standard Sobolev space $H^s(\mathbb{R}^n)\doteq H^s_{1}(\mathbb{R}^n)$).
For $s\in\mathbb{N}$, the norm in \eqref{normagamma} turns to be equivalent, {\it uniformly with respect to} $\gamma$, to the norm $||\cdot||_{H^s_{\gamma}(\mathbb{R}^n)}$ defined by
\begin{equation*}\label{derivate}
||u||^2_{H^s_{\gamma}(\mathbb{R}^n)}\doteq \sum\limits_{|\alpha|\le s}\gamma^{2(s-|\alpha|)}||\partial^{\alpha}u||^2_{L^2(\mathbb{R}^n)}\,,
\end{equation*}
{where we write $\partial^{\alpha}=\partial^{\alpha_1}_1\dots\partial^{\alpha_n}_n$ for the usual partial derivative corresponding to a multi-index $\alpha\in\mathbb{N}^n$.}

For functions defined over $Q^-_T$ we will consider the weighted Sobolev spaces $H^m_{\gamma}(Q^-_T)$ equipped with the $\gamma-$depending norm
\begin{equation*}\label{derivate2}
||u||^2_{H^m_{\gamma}(Q^-_T)}\doteq \sum\limits_{|\alpha|\le m}\gamma^{2(m-|\alpha|)}||\partial^{\alpha}u||^2_{L^2(Q^-_T)}\,.
\end{equation*}
Similar weighted Sobolev spaces will be considered for functions defined on $Q^-$.


\subsection{Conormal Sobolev spaces}
Let us introduce some classes of function spaces of Sobolev type, defined over the half-space $Q^+_T$.
For $j=0,\dots,3$, we set
$$
Z_0=\dt,\quad Z_1\doteq \sigma (x_1)\partial_1\,,\quad Z_j\doteq \partial_j\,,\,\,{\rm for}\,\,j= 2,3\,,
$$
where $\sigma (x_1)\in C^{\infty}(\mathbb{R}_+)$ is
a monotone increasing function such that $\sigma (x_1)=x_1$ in some neighborhood of
the origin and $\sigma (x_1)=1$ for $x_1$ large enough.
Then, for every multi-index $\alpha=(\alpha_0,\dots,\alpha_3)\in\mathbb{N}^4$, the {\it conormal} derivative $Z^{\alpha}$ is defined by
$$
Z^{\alpha}\doteq Z_0^{\alpha_0}\dots Z^{\alpha_3}_3\, .
$$
Given an integer $m\geq 1$, the {\it conormal Sobolev space} $H^m_{\tan}(Q^+_T)$ is defined as the set of functions $u\in L^2(Q^+_T)$ such that
$Z^\alpha u\in  L^2(Q^+_T)$, for all multi-indices $\alpha$ with $|\alpha|\le m$ (see \cite{moseBVP,moseIBVP}). Agreeing with the notations set for the usual Sobolev spaces, for $\gamma\ge 1$, $H^m_{\tan,\gamma}(Q^+_T)$ will denote the conormal space of order $m$ equipped with the $\gamma-$depending norm
\begin{equation}\label{normaconormale}
||u||^2_{H^{m}_{\tan,\gamma}(Q^+_T)}\doteq \sum\limits_{|\alpha|\le m}\gamma^{2(m-|\alpha|)}||Z^{\alpha}u||^2_{L^2(Q^+_T)}\,
\end{equation}
\noindent
and we have $H^m_{\tan}(Q^+_T)\doteq H^m_{\tan,1}(Q^+_T)$. Similar conormal Sobolev spaces with $\gamma$-depending norms will be considered for functions defined on $Q^+$.

We will use the same notation for spaces of scalar and vector-valued functions.

\section{The main result}\label{mainresult}

We are now in a position to state the main result of the paper. Recall that $\mathcal{U}=(q,u,h,S)$, where $u$ and $h$ were defined in \eqref{defcalU}.

%
\begin{theorem}\label{main}
Let $T>0$. Let the basic state \eqref{21} satisfies assumptions \eqref{22}--\eqref{27} and define
\begin{equation}\label{eq:mu}
\hat\mu\doteq  (\hat\Ec_1-\mep\hat v_3\hat\Hc_2+\mep\hat v_2\hat\Hc_3)_{|x_1=0}\,.
\end{equation}
We assume that the plasma velocity of the basic state at the boundary is smaller than the light speed at the boundary, namely \begin{equation}\label{hyp:velocity}
|\hat v|{, |\hat v_N|}<{\mep^{-1}} \qquad \mbox{on  }\omega_T,
\end{equation}
and assume also
\begin{equation}
|\hat{H} \times \hat{\mathcal{H}}|\geq \delta > 0 \qquad \mbox{on  }\omega_T,\label{41}
\end{equation}
where $\delta$ is an arbitrary fixed constant. Then there exists a constant~$\mu^*$ and there exists $\gamma_0\ge1$ such that, for all $|\hat\mu|\leq\mu^*$,  $\gamma\ge\gamma_0$ and for all $\mathcal{F}_\gamma \in H^1_{\tan,\gamma}(Q^+_T)$, vanishing in the past, namely for $t<0$, any solution $(\mathcal{U}_\gamma,\Wc_\gamma,\varphi_\gamma)\in H^1_{\tan,\gamma}(Q^+_T)\times H^1_{\gamma}(Q^-_T)\times H^1_\gamma(\omega_T)$ to problem \eqref{34'new}, with trace $(q_\gamma,u_{1\gamma},h_{1\gamma})|_{\omega_T}\in {H^{1/2}_\gamma(\omega_T)}$, 
obeys the a priori estimate
\begin{multline}
\gamma\left(\|\mathcal{U}_\gamma\|^2_{H^1_{\tan,\gamma}(Q^+_T)}+\|\Wc_\gamma\|^2_{H^{1}_\gamma(Q^-_T)}
+\|(q_\gamma,u_{1\gamma},h_{1\gamma})|_{\omega_T}\|^2_{H^{1/2}_\gamma(\omega_T)}\right)
\\
+\gamma^2\|\varphi_\gamma\|^2_{H^1_\gamma(\omega_T)}
\leq  \frac{C}{\gamma}\|\mathcal{F}_\gamma\|^2_{H^1_{\tan,\gamma}(Q^+_T)},
\label{54}
\end{multline}
where we have set $\mathcal{U}_\gamma=e^{-\gamma t}\,\mathcal{U}, \mathcal{H}_\gamma=e^{-\gamma t}\, \mathcal{H}, \varphi_\gamma= e^{-\gamma t}\, \varphi$
and so on. Here $C=C(\kappa,T,\delta)>0$ is a constant independent of the data $\mathcal{F}$ and $\gamma$.
\label{t1}
\end{theorem}
\begin{remark}\label{}
In \cite{SeTr}, where the pre-Maxwell dynamics in the vacuum side is considered, instead of the Maxwell equations, it is shown that the stability condition \eqref{41} is sufficient by itself for the well-posedness of the problem. Here we also need to impose \eqref{hyp:velocity} and the smallness condition $|\hat\mu|\leq\mu^*$. The present situation with more restrictive stability conditions is more similar to the relativistic case studied in \cite{trakhinin12}, with the difference that, in order to prevent violent instabilities, we don't need to assume that the plasma expands into the vacuum, as in \cite{trakhinin12}.

{Let us note as well that, in real problems, $\mep$ is very small if compared with the velocity of the basic state, so that \eqref{hyp:velocity} is always satisfied. Alternatively, given the basic state, we can choose $\bar{v}$ (see the definition of $\mep$) so that \eqref{hyp:velocity} holds.}

{Moreover, the smallness of $\mep$ says essentially that $\hat\mu$ is small provided that $\Ec_1$ be small, i.e. the electric field be sufficiently weak. Thus, when the electric field is weak, it is justified to neglect the displacement current and consider the classical plasma-vacuum non-relativistic model with pre-Maxwell dynamics. On the other hand, when the electric field is sufficiently strong, it is natural to expect that the plasma-vacuum interface problem could be ill-posed, see \cite{trakhinin12}. }

\end{remark}
\begin{remark}
{Independently on Hypothesis~\ref{hyp:invertibility},} assumption~\eqref{hyp:velocity} also {directly} implies that
\begin{equation}\label{eq:phitN}
\mep |\hat \varphi_t| < |\hat N| \,,
\end{equation}
where~$\hat N$ is as in~\eqref{eq:hatN}, {since~$\mep\,|\hat \varphi_t|=\mep\,|\hat v_N|\leq \mep\,|\hat v|\,|\hat N|<|\hat N|$.}
\end{remark}
%
%
%

\section{Proof of Theorem \ref{main}}\label{basic}

To prove the a priori estimate for problem~\eqref{34'new} given in Theorem~\ref{main}, we will extend our problem to the spaces~$Q^\pm, \omega$ given in~\eqref{defQ}.
\subsection{The extended boundary value problem}
Assuming that all coefficients and data appearing in \eqref{34'new} and~\eqref{newmaxwell} are extended for all times to the whole real line (for such standard procedure we refer the reader to \cite{lionsmagenes1,mosetre09}), let us consider the boundary value problem (recall the definition of $Q^\pm,\omega$ in \eqref{defQ})
\begin{subequations}\label {77'}
\begin{align}
\displaystyle \hat{\mathcal{A}}_0\partial_t{\mathcal{U}}+\sum_{j=1}^{3}(\hat{\mathcal{A}}_j+{{E}}_{1j+1})\partial_j{\mathcal{U}}+
\hat{\mathcal C}'{\mathcal{U}}=\mathcal{F} ,
 \qquad &\mbox{in}\ Q^+,\label{77'a}
\\
\mep{B}_0\partial_t\Ws+\sum_{j=1}^3B_j\partial_j\Ws +\mep{B}_4\Ws=0  \qquad &\mbox{in}\ Q^-, \label{78'}
\\
\varphi_t+\hat{v}_2\varphi_2+\hat{v}_3\varphi_3-
\varphi\partial_1\hat{v}_{N}-u_1 =0, \qquad&\label{79'} \\
{q} +  [ \partial_1\hat{q}]\varphi-\hat\hg\cdot\Hs+\hat\Es\cdot\eg=0,  \qquad&\label{80'}
\\
\Es_2 - \mep \p_t(\varphi\hat{\Hc}_3) + \p_2(\varphi\hat{\Ec}_1) = 0, \qquad&\label{81'} \\
\Es_3 + \mep \p_t(\varphi\hat{\Hc}_2) + \p_3(\varphi\hat{\Ec}_1) =0
 \qquad&\mbox{on}\ \omega, \label{82'}
\\
(\mathcal{U},\Ws,\varphi)=0\qquad &\mbox{for}\ t<0.\label{83g'}
\end{align}
\end{subequations}
Since problem \eqref{77'} looks similar to a corresponding one in relativistic MHD \cite{trakhinin12},
for the deduction of estimate \eqref{54} we use the same ideas as in \cite{trakhinin12}.  As in \cite{SeTr,trakhinin12} we will need a secondary symmetrization of the transformed Maxwell equations in vacuum.

\subsection{ A secondary symmetrization}
In order to show how to get the secondary symmetrization, for the sake of simplicity we first consider a planar unperturbed interface, i.e. the case $\hat{\varphi}\equiv 0$. For this case {\eqref{36'bnew}, \eqref{85}} become
\begin{equation}
\mep \partial_t \Wc +\sum_{j=1}^3B_j \partial_j\Wc=0,
\label{13'}
\end{equation}
\begin{equation}
\dv\Hc=0,\quad \dv\Ec=0,
\label{14}
\end{equation}
that is, the classical Maxwell system~\eqref{eq:Maxwell}.
\\
We write for system \eqref{13'} the following secondary symmetrization (for a similar secondary symmetrization of the Maxwell equations in vacuum see \cite{SeTr,trakhinin12}):
\begin{equation}
\mep \mathfrak{B}_0\partial_t\Wc  +\sum_{j=1}^3\mathfrak{B}_0B_j\partial_j\Wc +R_1{\rm div}\,\mathcal{H}
+R_2{\rm div}\, E
 =\mep \mathfrak{B}_0\partial_t\Wc +\sum_{j=1}^3\mathfrak{B}_j\partial_j\Wc=0,
\label{20'}
\end{equation}
where
\begin{equation}
\begin{array}{ll}\label{B0e}
\mathfrak{B}_0=\left(\begin{array}{cccccc}
1 & 0 & 0& 0 & \nu_3 & -\nu_2 \\
0 & 1 & 0& -\nu_3 & 0 & \nu_1 \\
0 & 0 & 1& \nu_2 & -\nu_1 & 0 \\
0 & -\nu_3 & \nu_2& 1 & 0 & 0 \\
\nu_3 & 0 & -\nu_1& 0 & 1 & 0 \\
-\nu_2 & \nu_1 & 0& 0 & 0 & 1
\end{array} \right),
\end{array}
\end{equation}
\begin{gather*}
\mathfrak{B}_1=
\left(\begin{array}{cccccc}
\nu_1 & \nu_2 & \nu_3& 0 & 0 & 0 \\
\nu_2 & -\nu_1 & 0& 0 & 0 & -1 \\
\nu_3 & 0 & -\nu_1& 0 & 1 & 0 \\
0 & 0 & 0& \nu_1 & \nu_2 & \nu_3 \\
0 & 0 & 1& \nu_2 & -\nu_1 & 0 \\
0 & -1 & 0& \nu_3 & 0 & -\nu_1
\end{array} \right),
\\
\mathfrak{B}_2=
\left(\begin{array}{cccccc}
-\nu_2 & \nu_1 & 0& 0 & 0 & 1 \\
\nu_1 & \nu_2 & \nu_3& 0 & 0 & 0 \\
0 & \nu_3 & -\nu_2& -1 & 0 & 0 \\
0 & 0 & -1& -\nu_2 & \nu_1 & 0 \\
0 & 0 & 0& \nu_1 & \nu_2 & \nu_3 \\
1 & 0 & 0& 0 & \nu_3 & -\nu_2
\end{array} \right),
\end{gather*}
\[
\mathfrak{B}_3=
\left(\begin{array}{cccccc}
-\nu_3 & 0 & \nu_1& 0 & -1 & 0 \\
0 & -\nu_3 & \nu_2& 1 & 0 & 0 \\
\nu_1 & \nu_2 & \nu_3& 0 & 0 & 0 \\
0 & 1 & 0& -\nu_3 & 0 & \nu_1 \\
-1 & 0 & 0& 0 & -\nu_3 & \nu_2 \\
0 & 0 & 0& \nu_1 & \nu_2 & \nu_3
\end{array} \right),
\quad
R_1=\left(\begin{array}{c} \nu_1 \\
\nu_2 \\
\nu_3 \\
0 \\
0 \\
0
\end{array} \right),\quad R_2=\left(\begin{array}{c}
0 \\
0 \\
0 \\
\nu_1 \\
\nu_2 \\
\nu_3
\end{array} \right).
\]
The arbitrary functions $\nu_i(t,x)$ will be chosen in appropriate way later on.
 It may be useful to notice that system \eqref{20'} can also be written as
 \begin{equation}
\begin{array}{ll}\label{secsym}
\displaystyle
(\mep\partial_t \mathcal{H}+\nabla\times \Ec) - \vec\nu\times
(\mep\partial_t \Ec-\nabla\times \mathcal{H})+
\vec\nu\,\dv\mathcal{H}=0,
\\
\\
\displaystyle
(\mep\partial_t \Ec-\nabla\times \mathcal{H}) + \vec\nu\times
(\mep\partial_t \mathcal{H}+\nabla\times \Ec)+
\vec\nu\, \dv\Ec=0,
\end{array}
\end{equation}
with the vector-function $\vec\nu = (\nu_1, \nu_2, \nu_3)$. The symmetric system \eqref{20'} (or \eqref{secsym}) is hyperbolic if $\mathfrak{B}_0>0$, i.e. for
\begin{equation}
\begin{array}{ll}\label{99}
|\vec\nu |<1.
\end{array}
\end{equation}
Since 
$$
\mbox{det}(\mathfrak{B}_1)=\nu_1^2\left(|\vec\nu|^2-1\right)^2,
$$
the boundary is noncharacteristic for system \eqref{20'} (or \eqref{secsym}) provided \eqref{99} and $\nu_1\not=0$ hold.


Consider now a nonplanar unperturbed interface, i.e., the general case when $\hat{\varphi}$ is not identically zero. Dealing with the variable~$\Wg$, we may write the system
\begin{equation}
{\mep M_0^\G\partial_t\Wg}+ \sum_{{j=1}}^3M_j^\G\partial_j\Wg+\mep M_4^\G\Wg=0,
\label{eq:Mgoth}
\end{equation}
where
\begin{equation}
\begin{array}{ll}\label{defmatricesgoth}
\ds M_j^\G=\frac{1}{(1+\hat\Psi_1)}\,K\mathfrak{B}_jK\tm \quad (j=0,2,3), \\
\ds M_1^\G=\frac{1}{1+\hat\Psi_1}\,K
\tilde{\mathfrak{B}}_1K\tm, \quad
\tilde{\mathfrak{B}}_1=\frac{1}{1+\hat\Psi_1}\Bigl(
\mathfrak{B}_1{-\mep\hat{\Psi}_t\mathfrak{B}_0}-\sum_{{k=2,3}}\hat{\Psi}_k\mathfrak{B}_k \Bigr),
\\
\ds {M}_4^\G=
K\left(\mep{\mathfrak{B}}_0\partial_t+\tilde{\mathfrak{B}}_1\partial_1+\mathfrak{B}_2 \partial_2+\mathfrak{B}_3 \partial_3+ \mep\mathfrak{B}_0{B}_4^\G\right)
\left( \frac{1}{1+\hat\Psi_1}\,K\tm \right).
\end{array}
\end{equation}
%
%

%
{Since}~$J$ is invertible (see {Hypothesis~\ref{hyp:invertibility}}), we may multiply~\eqref{eq:Mgoth} by
\[ \left((1+\hat\Psi_1)(K\tm)^{-1}J^{-1}\right)\tm \equiv \frac1{1-{\mep^2}\hat\Psi_t^2}\,J_1\tm K^{-1}\,, \]
on the lefthand side, obtaining the symmetric system
\begin{equation}
{M_0\partial_t\Ws} + \sum_{{j=1}}^3M_j\partial_j\Ws+ M_4\Ws=0,
\label{50"}
\end{equation}
where
\begin{equation}
\label{defmatrices}\begin{split}
&\ds M_j=\frac1{{(1+\hat\Psi_1)(1-\mep^2\hat\Psi_t^2)^2}}\,J_1\tm\mathfrak{B}_jJ_1 \quad (j=0,2,3), \\
&\ds M_1=\frac1{{(1+\hat\Psi_1)(1-\mep^2\hat\Psi_t^2)^2}}\,J_1\tm\tilde{\mathfrak{B}}_1J_1\,, 
\end{split}\end{equation}
and the precise expression of~$M_4$ is not important.
System \eqref{50"} originates from a linear combination of equations \eqref{newmaxwell} similar to \eqref{secsym}, namely from
\begin{equation}
\begin{array}{ll}\label{secsym2}
\displaystyle
(\mep\p_t\hg + \rot \Es) -
\hat\eta\left(\vec\nu\times\hat\eta^{-1}(\mep\p_t\eg - \rot \Hs)\right)+
\frac{\hat\eta\,\vec\nu}{1+\hat\Psi_1}\,\dv\hg=0\, ,
\\
\\
\displaystyle
(\mep\p_t\eg - \rot \Hs) +
\hat\eta\left(\vec\nu\times\hat\eta^{-1}(\mep\p_t\hg + \rot \Es)\right)+
\frac{\hat\eta\,\vec\nu}{1+\hat\Psi_1}\, \dv\eg=0\, .
\end{array}
\end{equation}
The equivalence of systems \eqref{36'bnew} and \eqref{secsym2} for every $\vec\nu\not=0$ follows as in~\cite{SeTr}. This is the same as the equivalence of \eqref{78'} and \eqref{50"}.
\begin{lemma}\label{equivalence}
Assume that systems \eqref{36'bnew} and \eqref{secsym2} have common initial data satisfying the constraints
$$
{\rm div}\,{\mathfrak{h}}=0,\quad {\rm div}\,{\mathfrak{e}}=0\qquad\mbox{in }\Omega^-\quad\mbox{for}\ t=0.
$$
Assuming that the corresponding Cauchy problems for \eqref{36'bnew} and \eqref{secsym2} have a unique classical solution on a time interval $[0,T]$, then these solutions coincide on $[0,T]$.
\end{lemma}
\begin{proof}
See \cite{SeTr}.
\end{proof}
Systems~\eqref{eq:Mgoth} and~\eqref{50"} are symmetric hyperbolic provided that~\eqref{99} holds. We compute
\begin{equation}
\begin{array}{ll}\label{determinant}
\mbox{det}(\mathfrak{B}_1)=\left(|\hat N|^2-{\mep^2}\hat\Psi_t^2\right)^2\left(\nu_1-\nu_2\hat\Psi_2 -\nu_3\hat\Psi_3-\mep\hat\Psi_t \right)^2\left(|\vec\nu|^2-1\right)^2 ,
\end{array}
\end{equation}
where the definition of~$\hat N$ is given in~\eqref{eq:hatN}. So the boundary is noncharacteristic if and only if \eqref{99} holds and $\nu_1\not=\nu_2\hat\varphi_2 +\nu_3\hat\varphi_3+\mep\hat\varphi_t$ at~$x_1=0$. Indeed, we recall that ${\mep^2}\hat\varphi_t^2<|\hat N|^2$, see~\eqref{eq:phitN}.
\\
In particular, the matrix~$M_1$ is of a special interest. We may write
\begin{equation}\label{eq:M1blocks}
M_1 = \frac1{(1+\hat\Psi_1)(1-{\mep^2}\hat\Psi_t^2)} \begin{pmatrix}
S & T \\
-T & S
\end{pmatrix}\,,
\end{equation}
where
{\tiny\begin{align*}
S
	 = \begin{pmatrix}
\frac{(\nu_1-\nu_2\,\hat\Psi_2-\nu_3\,\hat\Psi_3-\mep\hat\Psi_t) \,(|\hat N|^2-{\mep^2}\hat\Psi_t^2)}{(1+\hat\Psi_1)^2} & \nu_2\,\frac{|\hat N|^2-{\mep^2}\hat\Psi_t^2}{1+\hat\Psi_1}  & \nu_3\,\frac{\hat |N|^2-{\mep^2}\hat\Psi_t^2}{1+\hat\Psi_1} \\*[1em]
\nu_2\,\frac{|\hat N|^2-{\mep^2}\hat\Psi_t^2}{1+\hat\Psi_1} & -\nu_1-\nu_2\,\hat\Psi_2+\nu_3\,\hat\Psi_3+\mep\hat\Psi_t  & -(\nu_2\,\hat\Psi_3+\nu_3\,\hat\Psi_2)  \\*[1em]
\nu_3\,\frac{|\hat N|^2-{\mep^2}\hat\Psi_t^2}{1+\hat\Psi_1}  & -(\nu_2\,\hat\Psi_3+\nu_3\,\hat\Psi_2)  & -\nu_1+\nu_2\,\hat\Psi_2-\nu_3\,\hat\Psi_3+\mep\hat\Psi_t
\end{pmatrix}\,,\end{align*}}
\begin{align*}T
	= (1-\mep\nu_1\hat\Psi_t) B_1'\,.
\end{align*}
We fix
\begin{equation}\label{eq:nu1}
\nu_1 \doteq  \nu_2 \hat\varphi_2 + \nu_3\hat\varphi_3 + \mep \hat\varphi_t\,,
\end{equation}
so that on the boundary (we also recall $\hat\Psi_1=0$), we obtain
\begin{align}
\label{eq:S}
S(t,0,x')
	& = \begin{pmatrix}
0 & \nu_2\,(|\hat N|^2-{\mep^2}\hat\varphi_t^2)  & \nu_3\,(|\hat N|^2-{\mep^2}\hat\varphi_t^2) \\
\nu_2\,(|\hat N|^2-{\mep^2}\hat\varphi_t^2) & -2\nu_2\,\hat\varphi_2  & -(\nu_2\,\hat\varphi_3+\nu_3\,\hat\varphi_2)  \\
\nu_3\,(|\hat N|^2-{\mep^2}\hat\varphi_t^2)  & -(\nu_2\,\hat\varphi_3+\nu_3\,\hat\varphi_2)  & -2\nu_3\,\hat\varphi_3
\end{pmatrix}\,,\\
\label{eq:T}
T
	& = (1-{\mep^2}\hat\varphi_t^2-\mep\nu_2\hat\varphi_t\hat\varphi_2-\mep\nu_3\hat\varphi_t\hat\varphi_3) B_1'\,.
\end{align}
%
%
To prove estimate \eqref{54}, we need some computation.
\begin{lemma}\label{lem:MWW}
Let~$M_1$ be as in~\eqref{defmatrices}. Then
\begin{equation}\label{eq:WM1W}
\frac{1}{2}(M_1\Ws,\Ws)|_{\omega} = \Hs_3\Es_2-\Hs_2\Es_3 + \nu_2 (\Hs_2\hg_1+\Es_2\eg_1) + \nu_3 (\Hs_3\hg_1 + \Es_3\eg_1) \,.
\end{equation}
%
%
\end{lemma}
\begin{proof}
Thanks to~\eqref{eq:M1blocks}, \eqref{eq:S} and~\eqref{eq:T}, we may write
\begin{align*}
(1-{\mep^2}\hat\Psi_t)^2\,\frac{1}{2}(M_1\Ws,\Ws)|_{\omega} \\
\qquad = &(1-{\mep^2}\hat\varphi_t^2) (\Hs_3\Es_2-\Hs_2\Es_3) \\
	& + \nu_2 \Hs_2 \left\{ |\hat N|^2 \Hs_1 -\hat\varphi_2 \Hs_2 - \hat\varphi_3 \Hs_3 + \mep\hat \varphi_t\hat\varphi_2 \Es_3 - {\mep^2}\hat \varphi_t^2 \Hs_1 \right\} \\
	& + \nu_2 \Es_2 \left\{ |\hat N|^2 \Es_1 - \hat\varphi_2 \Es_2 - \hat\varphi_3 \Es_3 {-} \mep\hat \varphi_t\hat\varphi_2 \Hs_3 - {\mep^2}\hat \varphi_t^2 \Es_1 \right\} \\
	& + \nu_3 \Hs_3 \left\{ |\hat N|^2 \Hs_1 - \hat\varphi_2 \Hs_2 - \hat\varphi_3 \Hs_3 {-} \mep \hat \varphi_t\hat\varphi_3 \Es_2 - {\mep^2}\hat \varphi_t^2 \Hs_1 \right\} \\
	& + \nu_3 \Es_3 \left\{ |\hat N|^2 \Es_1 - \hat\varphi_2 \Es_2 - \hat\varphi_3 \Es_3 + \mep\hat \varphi_t\hat\varphi_3 \Hs_2 - {\mep^2}\hat \varphi_t^2 \Es_1 \right\}\,.
\end{align*}
Since~$\Hs_1=\Hc_1$ we may replace
\begin{align*}
|\hat N|^2 \Hs_1 & - \hat\varphi_2 \Hs_2 - \hat\varphi_3 \Hs_3 + \mep\hat \varphi_t\hat\varphi_2 \Es_3 - {\mep^2}\hat \varphi_t^2 \Hs_1 \\
	& = (1+\hat\varphi_2^2 + \hat\varphi_3^2) \Hc_1 - \hat\varphi_2(\Hc_2+\hat\varphi_2\Hc_1) - \hat\varphi_3(\Hc_3+{\hat\varphi_3\Hc_1}) \\
	& \quad + \mep\hat\varphi_t (-\hat\varphi_2\Ec_3+\hat\varphi_3\Ec_2) + \mep\hat\varphi_t\hat\varphi_2 (\Ec_3+\hat\varphi_3\Ec_1) - {\mep^2}\hat\varphi_t^2 (\Hc_1-\hat\varphi_2\Hc_2) \\
	& = \hg_1 + \mep\hat\varphi_t\hat\varphi_3\Es_2 - {\mep^2}\hat\varphi_t^2 \hg_1 \,.
\end{align*}
Similarly, we get
\begin{align*}
|\hat N|^2 \Es_1 - \hat\varphi_2 \Es_2 - \hat\varphi_3 \Es_3 {-} \mep\hat \varphi_t\hat\varphi_2 \Hs_3 - {\mep^2}\hat \varphi_t^2 \Es_1
	& = \eg_1 - \mep\hat\varphi_t\hat\varphi_3\Hs_2 - {\mep^2}\hat\varphi_t^2 \eg_1 \,, \\
|\hat N|^2 \Hs_1 - \hat\varphi_2 \Hs_2 - \hat\varphi_3 \Hs_3 {-} \mep\hat \varphi_t\hat\varphi_3 \Es_2 - {\mep^2}\hat \varphi_t^2 \Hs_1
	& = \hg_1 - \mep\hat\varphi_t\hat\varphi_2\Es_3 - {\mep^2}\hat\varphi_t^2 \hg_1 \,,\\
|\hat N|^2 \Es_1 - \hat\varphi_2 \Es_2 - \hat\varphi_3 \Es_3 + \mep\hat \varphi_t\hat\varphi_3 \Hs_2 - {\mep^2}\hat \varphi_t^2 \Es_1
	& = \eg_1 + \mep\hat\varphi_t\hat\varphi_2\Hs_3 - {\mep^2}\hat\varphi_t^2 \eg_1 \,.
\end{align*}
The proof of~\eqref{eq:WM1W} immediately follows.
\end{proof}
\begin{definition}
We take~$\vec\nu$ such that
\begin{equation}\label{eq:nuv}
\vec\nu = \mep \hat v\,,\qquad \text{on the boundary~$\omega$,}
\end{equation}
and we extend~$\vec\nu$ outside the boundary keeping~$|\vec\nu|<1$.
\end{definition}
Relation~\eqref{eq:nu1} follows from~\eqref{eq:nuv}, by virtue of the first part of~\eqref{24}, that is, $\hat v_N=\hat\varphi_t$.
\begin{lemma}
Let us denote
\[ {\mathcal{U}}_{\gamma} \doteq  e^{-\gamma t} \mathcal{U}\,, \qquad \Ws_\gamma \doteq  e^{-\gamma t} \Ws\,, \qquad \varphi_\gamma \doteq  e^{-\gamma t} \varphi\,, \]
for some~$\gamma\geq0$. Then, using~\eqref{eq:nuv} and the boundary conditions \eqref{79'}--\eqref{82'}, as well as the costraint~\eqref{87}, the quadratic form
\begin{equation}\label{eq:A}
\mathcal{A}\doteq  -\frac{1}{2}({{E}}_{12}{\mathcal{U}}_{\gamma},{\mathcal{U}}_{\gamma})|_{\omega} +\frac{1}{2}(M_1\Ws_{\gamma},\Ws_{\gamma})|_{\omega}
\end{equation}
may equivalently be written as
\begin{align}
\nonumber
\mathcal{A}
	& = \hat\mu\,\bigl(\varphi_3\Hs_2-\varphi_2\Hs_3+\mep(\varphi_t+\gamma\varphi)\eg_1\bigr) + \varphi \, \bigl\{\mep q\p_1\hat v_N + \mep{u_1[\p_1\hat q]} + \mep\varphi [\p_1\hat q]\p_1 \hat v_N \\
\label{eq:Aphi}
	& \qquad + (\Hs_3+\mep\hat v_2\eg_1) (\mep\p_t\hat{\Hc}_3 - \p_2 \hat{\Ec}_1) + (\Hs_2-\mep\hat v_3 \eg_1)(\mep\p_t \hat{\Hc}_2 +\p_3\hat{\Ec}_1) \\
& \qquad + \mep(\hat v_2\Hs_2+ \hat v_3\Hs_3)(\p_2 \hat{\Hc}_2 +\p_3 \hat{\Hc}_3) \bigr\}\,, \nonumber
\end{align}
where~$\hat \mu$ is as in~\eqref{eq:mu}.
\end{lemma}
\begin{proof}
For the sake of brevity, we omit the pedices. By virtue of Lemma~\ref{lem:MWW} {and~\eqref{eq:nuv}}, it holds
\[ \mathcal{A} = -q\,u_1 + \Hs_3\Es_2-\Hs_2\Es_3 + \mep(\hat v_2\Hs_2+ \hat v_3\Hs_3)\hg_1 + \mep(\hat v_2\Es_2+\hat v_3\Es_3)\eg_1 \,. \]
As a consequence of \eqref{81'}--\eqref{82'}, and using~\eqref{87}, we may replace
\begin{align*}
\hg_1 & = \p_2 (\varphi\hat{\Hc}_2) +\p_3 (\varphi\hat{\Hc}_3)\, , \\
\Es_2 & = \mep\p_t(\varphi\hat{\Hc}_3) - \p_2(\varphi\hat{\Ec}_1)\, , \\
\Es_3 & = - \mep\p_t(\varphi\hat{\Hc}_2) - \p_3(\varphi\hat{\Ec}_1) \, ,\\
\end{align*}
deriving
\begin{align*}
\mathcal{A}
	& = -q\,u_1 + \Hs_3(\mep\varphi_t\,\hat{\Hc}_3-\varphi_2\,\hat{\Ec}_1)+\Hs_2(\mep\varphi_t\,\hat{\Hc}_2+\varphi_3\,\hat{\Ec}_1) \\
& \qquad + \mep(\hat v_2\Hs_2+ \hat v_3\Hs_3)(\varphi_2\hat{\Hc}_2+\varphi_3\hat{\Hc}_3) \\
	& \qquad + \mep\bigl(\hat v_2(\mep\varphi_t\,\hat{\Hc}_3- \varphi_2\,\hat{\Ec}_1)-\hat v_3( \mep\varphi_t\,\hat{\Hc}_2 + \varphi_3\,\hat{\Ec}_1)\bigr)\eg_1\\
	& \qquad + \varphi\, \biggl\{ \Hs_3(\mep\p_t\hat{\Hc}_3 - \p_2 \hat{\Ec}_1) + \Hs_2(\mep\p_t \hat{\Hc}_2 +\p_3\hat{\Ec}_1) \\
& \qquad + \mep(\hat v_2\Hs_2+ \hat v_3\Hs_3)(\p_2 \hat{\Hc}_2 +\p_3 \hat{\Hc}_3) \\
	& \qquad \qquad \left. + \mep\bigl(\hat v_2(\mep\p_t \hat{\Hc}_3 - \p_2 \hat{\Ec}_1)-\hat v_3(\mep\p_t\hat{\Hc}_2 + \p_3 \hat{\Ec}_1)\bigr)\eg_1 \right\} \\
	& \qquad + \mep\gamma\varphi\,\left\{ \Hs_3\hat\Hc_3 + \Hs_2\hat{\Hc}_2 + \mep(\hat v_2\hat{\Hc}_3 -\hat v_3\hat{\Hc}_2)\eg_1\right\} = \mathrm{(I)}+\varphi\,\mathrm{(II)}+\mep\gamma\varphi\,\mathrm{(III)}\,.
\end{align*}
%
Since
\begin{align*}
\mathrm{(I)}+\mep\gamma\varphi\,\mathrm{(III)} = -q\,u_1
	& + \varphi_2 (-\hat\Ec_1\Hs_3+\mep\hat v_2\hat\Hc_2\Hs_2+\mep\hat v_3\hat\Hc_2\Hs_3-\mep\hat v_2\hat\Ec_1\eg_1) \\
	& + \varphi_3 (\hat\Ec_1\Hs_2+\mep\hat v_2\hat\Hc_3\Hs_2+\mep\hat v_3\hat\Hc_3\Hs_3-\mep\hat v_3\hat\Ec_1\eg_1) \\
	& + \mep\varphi_t (\hat\Hc_3\Hs_3+\hat\Hc_2\Hs_2+\mep(\hat v_2\hat\Hc_3-\hat v_3\hat\Hc_2)\eg_1) \\
	& + \mep\gamma\varphi (\Hs_3\hat\Hc_3 + \Hs_2\hat{\Hc}_2 + \mep(\hat v_2\hat{\Hc}_3 -\hat v_3\hat{\Hc}_2)\eg_1) \,,
\end{align*}
using~\eqref{79'}, i.e.
\begin{align} \label{bc.gammaphi}
\varphi_t + \hat v_2 \varphi_2 + \hat v_3 \varphi_3 + \gamma\varphi = \varphi\p_1 \hat v_N +u_1\,,
\end{align}
we may write
\begin{align*}
\mathrm{(I)}+\mep\gamma\varphi\,\mathrm{(III)}
	& = -q\,u_1 + \mep(\varphi\p_1 \hat v_N +u_1) (\hat\Hc_2\Hs_2+\hat\Hc_3\Hs_3-\hat\Ec_1\eg_1) \\
	& \qquad + (\hat\Ec_1-\mep\hat v_3\hat\Hc_2+\mep\hat v_2\hat\Hc_3)(\varphi_3\Hs_2-\varphi_2\Hs_3+\mep(\varphi_t+\gamma\varphi)\eg_1)\\
	& = -q\,u_1 + \mep(\varphi\p_1 \hat v_N +u_1) (q+[\p_1\hat q]\varphi) \\
 & \qquad+ \hat\mu (\varphi_3\Hs_2-\varphi_2\Hs_3+\mep(\varphi_t+\gamma\varphi)\eg_1) \\
	& = \mep\varphi (q\p_1\hat v_N + u_1[\p_1\hat q]) + \mep\varphi^2 [\p_1\hat q]\p_1 \hat v_N \\
& \qquad + \hat\mu\,(\varphi_3\Hs_2-\varphi_2\Hs_3+\mep(\varphi_t+\gamma\varphi)\eg_1)\,,
\end{align*}
where in the second equality we used~\eqref{80'}, and where~$\hat\mu$ is as in~\eqref{eq:mu}.
The proof of~\eqref{eq:Aphi} immediately follows.
\end{proof}

\subsection{The a priori estimate}

For the proof of our basic a priori estimate \eqref{54} we  will apply the energy method to the symmetric hyperbolic systems \eqref{77'a} and \eqref{50"}.
In the sequel $\gamma_0\ge1$ denotes a generic constant sufficiently large which may increase from formula to formula, and $C$ is a generic constant that may change from line to line.

First of all we provide some preparatory estimates.
In particular, to estimate the weighted conormal derivative  $Z_1=\sigma\partial_1$ of $\mathcal{U}$ (recall the definition \eqref{normaconormale} of the $\gamma$-dependent norm of $H^1_{\tan,\gamma}$) we do not need any boundary condition because the weight $\sigma$ vanishes on $\omega$. Applying to system \eqref{77'a} the operator $Z_1$ and using standard arguments of the energy method\footnote{
We multiply $Z_1$\eqref{77'a} by $e^{-\gamma t}\,Z_1\mathcal{U}_\gamma$ and integrate by parts over $Q^+$, then we use the Cauchy-Schwarz inequality.
} yields the inequality
\begin{equation}\begin{split}
& \gamma \|Z_1\mathcal{U}_\gamma\|^2_{L^2(Q^+)} \\
& \qquad \leq \frac{C}{\gamma}\left\{ \|\mathcal{F}_\gamma\|^2_{H^1_{\tan,\gamma}(Q^+)}+\|\mathcal{U}_\gamma\|^2_{H^1_{\tan,\gamma}(Q^+)}
+\|{E}_{12}\partial_1\mathcal{U}_\gamma\|^2_{L^2(Q^+)} \right\}
,
\label{126}\end{split}
\end{equation}
for $\gamma\ge \gamma_0$. On the other hand, directly from the equation \eqref{77'a} we have
\begin{equation}
\begin{array}{ll}\label{126'}
\|{E}_{12}\partial_1\mathcal{U}_\gamma\|^2_{L^2(Q^+)}\le C
\left\{ \|\mathcal{F}_\gamma\|^2_{L^2(Q^+)}+\|\mathcal{U}_\gamma\|^2_{H^1_{\tan,\gamma}(Q^+)} \right\},
\end{array}
\end{equation}
where $C$ is independent of $\gamma$.
Thus from \eqref{126}, \eqref{126'} we get
\begin{equation}
\gamma \|Z_1\mathcal{U}_\gamma\|^2_{L^2(Q^+)} \leq \frac{C}{\gamma}\left\{ \|\mathcal{F}_\gamma\|^2_{H^1_{\tan,\gamma}(Q^+)}+\|\mathcal{U}_\gamma\|^2_{H^1_{\tan,\gamma}(Q^+)} \right\}
, \quad \gamma\ge \gamma_0,
\label{126''}
\end{equation}
where $C$ is independent of $\gamma$.
Furthermore, using the special structure of the boundary matrix in \eqref{77'a} (see \eqref{a10}) and the divergence constraint
\eqref{84}, we may estimate the normal derivative of the noncharacteristic part $\mathcal{U}_{n\gamma}=e^{-\gamma t}(q,u_1,h_1)$ of the ``plasma'' unknown $\mathcal{U}_{\gamma}$:
\begin{equation}
\|\partial_1 \mathcal{U}_{n\gamma} \|^2_{L^2(Q^+)} \leq  C\left\{ \|\mathcal{F}_\gamma\|^2_{L^2(Q^+)} +\|\mathcal{U}_\gamma\|^2_{H^1_{\tan,\gamma}(Q^+)} \right\},
\label{130}
\end{equation}
where $C$ is independent of $\gamma$.
In a similar way we wish to express the normal derivative of $\Ws$ through its tangential derivatives. Here it is convenient to use system \eqref{78'} rather than \eqref{50"}. We find from \eqref{78'} an explicit expression for the normal derivatives of $\Hs_2,\Hs_3,\Es_2,\Es_3$. An explicit expression for the normal derivatives of $\Hs_1,\Es_1$ is found through the divergence constraints \eqref{85}.
Thus we can estimate the normal derivatives of all the components of $\Ws$ through its tangential derivatives:
\begin{equation}\begin{split}
& \|\partial_1\Ws_\gamma\|^2_{L_2(Q^-)} \\
& \qquad \leq C
\left\{\gamma^2 \|\Ws_\gamma \|^2_{L_2(Q^-)}+\|\partial_t\Ws_\gamma \|^2_{L_2(Q^-)}+
\sum_{k=2}^3\|\partial_k\Ws_\gamma \|^2_{L_2(Q^-)} \right\}
,\label{127}\end{split}
\end{equation}
where $C$ does not depend on $\gamma$.

As for the front function $\varphi$ we easily obtain from \eqref{79'} the $L^2$ estimate
\begin{equation}
\gamma\|\varphi_{\gamma}\|^2_{L^2(\omega)}
\\
 \leq \frac{C}{\gamma}
\|u_{1\gamma}\|^2_{L^2(\omega)}, \quad   \gamma\ge \gamma_0,
\label{phi}
\end{equation}
where $C$ is independent of $\gamma$.
Furthermore, thanks to our basic assumption \eqref{41}\footnote{Under the conditions $\hat{H}_N=\hat{\mathcal{H}}_N=0$ one has
$|\hat{H}\times\hat{\mathcal{H}}|^2=(\hat{H}_2\hat{\mathcal{H}}_3
-\hat{H}_3\hat{\mathcal{H}}_2)^2\langle\nabla'\hat\varphi\rangle^2$ on $\omega$,
where we have set $\langle\nabla'\hat\varphi\rangle\doteq (1+|\hat\varphi_2|^2+|\hat\varphi_3|^2)^{1/2}$.}
we can resolve \eqref{86}, \eqref{87} and \eqref{79'} for the space-time gradient
\begin{align*}
\nabla_{t,x'}\varphi_\gamma =(\partial_t\varphi_\gamma,\partial_2\varphi_\gamma,\partial_3\varphi_\gamma)\, :
\end{align*}
\begin{equation}
\nabla_{t,x'}\varphi_{\gamma}=\hat{a}_1{h}_{1\gamma}+\hat{a}_2{\mathfrak{h}}_{1\gamma} +\hat{a}_3{u}_{1\gamma}+\hat{a}_4\varphi_\gamma+\gamma\hat{a}_5\varphi_\gamma ,
\label{12a}
\end{equation}
where the vector-functions $\hat{a}_{\alpha}={a}_{\alpha}(\hat{U}_{|\omega},\hat{\mathcal{H}}_{|\omega})$ of coefficients can be easily written in explicit form. From \eqref{12a} we get
\begin{equation}
\begin{array}{ll}\label{gradphi}
\|\nabla_{t,x'}\varphi_{\gamma}\|_{L^2(\omega)}\leq
C\left( \|\mathcal{U}_{n\gamma}|_\omega\|_{L^2(\omega)}+\|\Ws_{\gamma}|_\omega \|_{L^2(\omega)}+ \gamma\|\varphi_\gamma\|_{L^2(\omega)}
\right).
\end{array}
\end{equation}

Now we prove an $L^2$ energy estimate for $(\mathcal{U},\Ws)$. We multiply \eqref{77'a} by $e^{-\gamma t}\,\mathcal{U}_\gamma$ and \eqref{50"} by $e^{-\gamma t}\,\Ws_\gamma$, integrate by parts over $Q^\pm$, then we use the Cauchy-Schwarz inequality.
If we set
\begin{align*}
I = \int_{Q^+}(\hat{\mathcal A}_0{\mathcal{U}}_\gamma,{\mathcal{U}}_\gamma)\,dxdt+
\int_{Q^-}(M_0\Ws_\gamma,\Ws_\gamma)dxdt\, ,
\end{align*}
we easily obtain
\begin{align}
\gamma I+ \int_{\omega}\mathcal{A}\,dx'dt
 \leq C
\left\{ \frac{1}{\gamma}\|\mathcal{F}_\gamma\|^2_{L^2(Q^+)}+\|\mathcal{U}_\gamma\|^2_{L^2(Q^+)} +\|\Ws_\gamma\|^2_{L^2(Q^-)} \right\},
\label{19a0}
\end{align}
where~$\mathcal{A}$ is as in~\eqref{eq:A} and the constant $C$ in \eqref{19a0} is uniform with respect to $\gamma$.
%
%
Now observe that, if we temporarily omit $\gamma$ subscripts, for the term with $\hat\mu$ in $\mathcal A$, we have
\begin{align*}
& \hat\mu\,(\varphi_3\Hs_2-\varphi_2\Hs_3+(\varphi_t+\gamma\varphi)\eg_1) \\
& \qquad = \p_t (\hat\mu\varphi\eg_1) - \p_2(\hat\mu\varphi\Hs_3) + \p_3(\hat\mu\varphi\Hs_2) + 2\hat\mu\gamma\varphi\eg_1\\
& \qquad \qquad  -\hat\mu\varphi(\p_t\eg_1+\gamma\eg_1+\p_3\Hs_2-\p_2\Hs_3) + \text{l.o.t.}\\
& \qquad = \p_t (\hat\mu\varphi\eg_1) - \p_2(\hat\mu\varphi\Hs_3) + \p_3(\hat\mu\varphi\Hs_2)  + 2\hat\mu\gamma\varphi\eg_1 + \text{l.o.t.} \\
& \qquad = \tilde{\mathcal{A}} + 2\hat\mu\gamma\varphi\eg_1 + \text{l.o.t.}
\end{align*}
(where l.o.t. is the sum of lower-order terms) by exploiting the first component of $\p_t \eg_\gamma + \gamma\eg_{1,\gamma}-\rot\Hs_\gamma = 0$; using this argument and~\eqref{eq:Aphi} in \eqref{19a0}, and integrating $\tilde{\mathcal{A}}$, we obtain
\begin{equation}\begin{split}
 & \gamma I + 2\gamma\int_\omega \hat\mu \varphi_\gamma \eg_{1,\gamma} \, dx' dt \\
 & \qquad \leq \frac{C}{\gamma}\left\{\|\mathcal{F}_\gamma\|^2_{L^2(Q^+)}
 + \|\mathcal{U}_{n\gamma}|_\omega\|^2_{L^2(\omega)}+ \|\Ws_{\gamma}|_\omega\|^2_{L^2(\omega)}
 \right\} \label{L2a}
 \\
 & \qquad \qquad +C\left( \|\mathcal{U}_\gamma\|^2_{L^2(Q^+)}+\|\Ws_\gamma\|^2_{L^2(Q^-)}\right)
 +  \gamma \|\varphi_{\gamma}\|^2_{L^2(\omega)} \, ,
\end{split}\end{equation}
where $C$ is independent of $\gamma$.


Now we want to derive the a priori estimate for tangential derivatives. Differentiating systems \eqref{77'a} and \eqref{50"} with respect to $x_0=t$, $x_2$ or $x_3$, using standard arguments of the energy method, and applying \eqref{130}, \eqref{127}, gives the energy inequality
\begin{align}
\gamma I_\ell
+ \int_{\omega}\mathcal{A}_{\ell}\,dx'dt
\leq \frac{C}{\gamma}
\left\{ \|\mathcal{F}_\gamma\|^2_{H^1_{\tan,\gamma}(Q^+)}+\|\mathcal{U}_\gamma\|^2_{H^1_{\tan,\gamma}(Q^+)} +\|W_\gamma\|^2_{H^1_{\gamma}(Q^-)} \right\},
\label{19a}
\end{align}
where $\ell =0,2,3$, and where we have denoted
\begin{align*}
I_\ell = \int_{Q^+}(\hat{A}_0 Z_{\ell}\Uc_\gamma,Z_{\ell}\Uc_\gamma)\,dxdt+
\int_{Q^-}(M_0Z_{\ell}\Ws_\gamma,Z_{\ell}\Ws_\gamma)dxdt
\end{align*}
and
\begin{equation*}
\begin{array}{ll}\label{19b}
\ds \mathcal{A}_{\ell}=-\frac{1}{2}({E}_{12}Z_{\ell}\Uc_{\gamma},Z_{\ell}
\Uc_{\gamma})|_{\omega}+\frac{1}{2}(M_1Z_{\ell}\Ws_{\gamma},Z_{\ell}
\Ws_{\gamma})|_{\omega}\,.
\end{array}
\end{equation*}
%
Using again~\eqref{eq:Aphi}, we obtain (for simplicity we drop again the index $\gamma$)
\begin{align}\begin{split}
\mathcal{A}_{\ell}=&
\hat\mu\,\bigl((Z_{\ell}\varphi_3)Z_{\ell}\Hs_2-(Z_{\ell}\varphi_2)Z_{\ell}\Hs_3+(Z_{\ell}\varphi_t+\gamma Z_{\ell}\varphi)Z_{\ell}\eg_1\bigr) \\
& \qquad + Z_{\ell}\varphi \big\{ [ \partial_1\hat{q}]\,(Z_{\ell}u_1+\varphi(\partial_1\hat{v}_N)Z_{\ell}) + \partial_1\hat{v}_NZ_{\ell}q \\
& \qquad + (  \partial_t\hat{\mathcal{H}}_3 - \partial_2\hat{E}_1)(Z_{\ell}\Hs_{3} + \hat{v}_2 Z_{\ell}\eg_1)\\
 & \qquad + ( \partial_t\hat{\mathcal{H}}_2 + \partial_3\hat{E}_1)(Z_{\ell}\Hs_{2} - \hat{v}_3 Z_{\ell}\eg_1) \\ & \qquad +(\partial_2\hat{\mathcal{H}}_2 + \partial_3\hat{\mathcal{H}}_3 ) (\hat{v}_2Z_{\ell}\Hs_{2} + \hat{v}_3Z_{\ell}\Hs_{3}   )\big\}\\
& \qquad +{\rm l.o.t.}\,\! ,
\quad \mbox{on  }{\omega}.
\label{Al}
\end{split}\end{align}
 The only exception being the first term, using \eqref{12a} we reduce the above terms to those like
\[
\hat{c}\,h_1Z_{\ell}u_1,\quad \hat{c}\, {h}_1Z_{\ell}\varphi,
\quad
\hat{c}\, {h}_1Z_{\ell}\Hs_j,\quad \hat{c}\, {h}_1Z_{\ell}\Es_j
,\quad\dots
\quad \mbox{on  }{\omega}
,\]
terms as above with $\mathfrak{h}_1,u_1$ instead of $h_1$, or even ``better'' terms like
$$
\gamma\hat{c}\varphi Z_{\ell}{u}_1,\quad
\gamma\hat{c}\varphi Z_{\ell}\varphi.$$ Here and below $\hat{c}$ is the common notation for a generic coefficient depending on the basic state \eqref{21}.
By integration by parts such ``better'' terms can be reduced to the above ones and terms of lower order.
\\
The terms like $\hat{c}\,h_1Z_{\ell} u_{1|x_1=0}$ are estimated by passing to the volume integral and integrating by parts:
\begin{multline*}
\int_{\omega}\hat{c}\,h_1Z_{\ell}u_{1|x_1=0}\,{\rm d}x'\,{\rm d}t
=-\int_{Q^+}\partial_1\bigl(\tilde{c}h_1Z_{\ell}u_1\bigr){\rm d}x\,{\rm d}t \\
=\int_{Q^+}\Bigl\{(Z_{\ell}\tilde{c})h_1(\partial_1u_1)+\tilde{c}(Z_{\ell}h_1)\partial_1u_1  -(\partial_1\tilde{c})h_1
Z_\ell u_1
-\tilde{c}(\partial_1h_1)Z_{\ell}u_1
\Bigr\}{\rm d}x\,{\rm d}t
,
\end{multline*}
where $\tilde{c}|_{x_1=0}=\hat{c}$. Estimating the righthand side by the H\"older's inequality and \eqref{130} gives
\begin{equation}
\begin{array}{ll}\label{unacaso}
\ds\left|
\int_{\omega}\hat{c}\,h_1Z_{\ell}u_{1|x_1=0}\,{\rm d}x'\,{\rm d}t
\right| \leq
  C\left\{ \|\mathcal{F}_\gamma\|^2_{L^2(Q^+)} +\|\mathcal{U}_\gamma\|^2_{H^1_{\tan,\gamma}(Q^+)} \right\}.
\end{array}
\end{equation}
Terms like $\gamma\hat c \varphi Z_\ell\varphi$ can be handled similarly. Note that, passing to the volume integral and integrating by part yields, among other things, the term
\begin{align} \label{est.gammaterm}
\gamma\left| \int_{Q^+} \tilde c (Z_\ell\varphi) \p_1 u_1 \, dx dt \right| \leq C\gamma \| Z_\ell\varphi \|^2_{L^2(Q^+)} + C\gamma \| \p_1 u_1 \|^2_{L^2(Q^+)}\, .
\end{align}
In the same way, we estimate the other similar terms $  \hat{c}\, {h}_1Z_{\ell}\Hs_j, \hat{c}\, {h}_1Z_{\ell}\Es_j
,$ etc. Notice that we only need to estimate normal derivatives either of components of $\mathcal{U}_{n\gamma}$ or $\Ws_{\gamma}$. For terms like $\hat{c}\,{\mathfrak{h}}_{1}Z_{\ell}u_1, \hat{c}\,{\mathfrak{h}}_{1}Z_{\ell}\Es_j$, etc. we use \eqref{127} instead of \eqref{130}.
We treat the terms like $\hat{c}\,h_{1|x_1=0}Z_{\ell} \varphi$  by substituting \eqref{12a} again:
\begin{align}\begin{split}\label{unacaso2}
&\left|\int_{\omega}\hat{c}\,h_1Z_{\ell}\varphi\,{\rm d}x'\,{\rm d}t\right|
=\left|\int_{\omega}\hat{c}\,h_1\Bigl(\hat{a}_1{h}_{1}+\hat{a}_2{\mathfrak{h}}_{1} +\hat{a}_3{u}_{1}+\hat{a}_4\varphi+\gamma\hat{a}_5\varphi
\Bigr){\rm d}x\,{\rm d}t \right|
\\
& \qquad \le
C\left( \|\mathcal{U}_{n}|_\omega\|^2_{L^2(\omega)}+\|\Ws|_\omega\|^2_{L^2(\omega)} +\gamma^2\|\varphi\|^2_{L^2(\omega)}\right)\\
& \qquad \leq C\left( \|\mathcal{U}_{n}|_\omega\|^2_{L^2(\omega)}+\|\Ws|_\omega\|^2_{L^2(\omega)}\right) .
\end{split}\end{align}
Combining \eqref{19a}, \eqref{unacaso}, \eqref{est.gammaterm}, \eqref{unacaso2} and similar inequalities for the other terms of \eqref{Al}, and proceeding similarly as above for the term with $\hat\mu$, yields (we restore the index $\gamma$)
\begin{align}\begin{split}
& \gamma I_\ell + 2\gamma\int_\omega \hat\mu (Z_\ell \varphi_\gamma) Z_\ell \eg_{1,\gamma} \, dx' dt \\
& \qquad  \leq C\Big\{\frac{1}{\gamma}\|\mathcal{F}_\gamma\|^2_{H^1_{\tan,\gamma}(Q^+)}
 +\|\mathcal{U}_\gamma\|^2_{H^1_{\tan,\gamma}(Q^+)}
 +\|\Ws_\gamma\|^2_{H^1_{\gamma}(Q^-)}
\\
 & \qquad +\gamma \left( \|\mathcal{U}_{n\gamma}|_\omega\|^2_{L^2(\omega)} +  \|\Ws_{\gamma}|_\omega\|^2_{L^2(\omega)} \right)
\Big\},
 \quad  \gamma\ge \gamma_0,
\label{Zell}
\end{split}\end{align}
where $C$ is independent of $\gamma$.
To treat the boundary integral in the lefthand side of \eqref{Zell}, which has no definite sign, we argue as in \cite{trakhinin12}. Let us recall that $\mu^*$ is a constant such that $|\hat\mu|\leq\mu^*$.
By substituting \eqref{12a} and proceeding as for \eqref{unacaso}, we deal with
\begin{align*}
& 2\gamma\left| \int_\omega \hat\mu (Z_\ell \varphi_\gamma) Z_\ell \eg_{1,\gamma} \, dx' dt \right| \\
& \qquad \leq
 C\gamma \mu^* \left\{ \|\mathcal{F}_\gamma\|^2_{L^2(Q^+)} +\|\mathcal{U}_\gamma\|^2_{H^1_{\tan,\gamma}(Q^+)} + \| \Ws_\gamma\|^2_{H^1_\gamma(Q^-)}  \right\}\\
 & \qquad \leq \frac{C}{\gamma} \mu^* \|\mathcal{F}_\gamma\|^2_{H^1_{\tan,\gamma}(Q^+)} +C\gamma\mu^* \left\{ \|\mathcal{U}_\gamma\|^2_{H^1_{\tan,\gamma}(Q^+)} + \| \Ws_\gamma\|^2_{H^1_\gamma(Q^-)}  \right\} .
\end{align*}
Therefore, we obtain
\begin{align}\begin{split}
\gamma I_\ell
 \leq& C\frac{1+\mu^*}{\gamma}\|\mathcal{F}_\gamma\|^2_{H^1_{\tan,\gamma}(Q^+)}\\
 & \qquad +C(1+\gamma\mu^*)\Bigl\{\|\mathcal{U}_\gamma\|^2_{H^1_{\tan,\gamma}(Q^+)}
 +\|\Ws_\gamma\|^2_{H^1_{\gamma}(Q^-)} \Bigr\} \\
 & \qquad +C\gamma \left( \|\mathcal{U}_{n\gamma}|_\omega\|^2_{L^2(\omega)} +  \|\Ws_{\gamma}|_\omega\|^2_{L^2(\omega)} \right) ,
 \quad  \gamma\ge \gamma_0\, .
\label{Zell2}
\end{split}\end{align}
Then, if we assume that $\mu^*$ is small enough (so that, for instance, $C\mu^*\leq 1/2$), from \eqref{126''}, \eqref{127}, \eqref{L2a} and \eqref{Zell2} we deduce
\begin{align}\begin{split}
& \gamma\left( \|\mathcal{U}_\gamma\|^2_{H^1_{\tan,\gamma}(Q^+)}
+\|\Ws_\gamma\|^2_{H^1_{\gamma}(Q^-)}\right)\\
 \leq& C\Big\{\frac{1}{\gamma}\|\mathcal{F}_\gamma\|^2_{H^1_{\tan,\gamma}(Q^+)}
 +\|\mathcal{U}_\gamma\|^2_{H^1_{\tan,\gamma}(Q^+)}
 +\|\Ws_\gamma\|^2_{H^1_{\gamma}(Q^-)}
\\ & \qquad+\gamma \left( \|\mathcal{U}_{n\gamma}|_\omega\|^2_{L^2(\omega)} +  \|\Ws_{\gamma}|_\omega\|^2_{L^2(\omega)} \right)
\Big\},
 \quad \gamma\ge \gamma_0,
\label{H1}
\end{split}\end{align}
where $C$ is independent of $\gamma$ (note that the second term in \eqref{L2a} is of lower order with respect to the analog one in \eqref{Zell2}).
%
%
We need the following estimates for the trace of $\mathcal{U}_n,{\Ws}$.
\begin{lemma}\label{interpol}
The functions $\mathcal{U}_n,\Ws$ satisfy
\begin{equation}
\begin{array}{ll}\label{interpol1}
\gamma\|\mathcal{U}_{n\gamma}|_\omega\|^2_{L^2(\omega)}
+
\|\mathcal{U}_{n\gamma}|_\omega\|^2_{H^{1/2}_{\gamma}(\omega)}
\le C\left(\|\mathcal{F}_\gamma\|^2_{L^2(Q^+)}
 +\|\mathcal{U}_\gamma\|^2_{H^1_{\tan,\gamma}(Q^+)}
\right),

\end{array}
\end{equation}
\begin{equation}
\begin{array}{ll}\label{interpol2}

\gamma  \|\Ws_{\gamma}|_\omega\|^2_{L^2(\omega)}
+
\|\Ws_{\gamma}|_\omega\|^2_{H^{1/2}_{\gamma}(\omega)}
\le C  \|\Ws_\gamma\|^2_{H^1_{\gamma}(Q^-)}.

\end{array}
\end{equation}
\end{lemma}
The proof of Lemma \ref{interpol} is completely analogous to the proof of Lemma~$13$ in~\cite{SeTr}. Substituting \eqref{interpol1}, \eqref{interpol2} in \eqref{H1} and taking $\gamma_0$ large enough yields
\begin{equation}
\gamma\left( \|\mathcal{U}_\gamma\|^2_{H^1_{\tan,\gamma}(Q^+)}
+\|\Ws_\gamma\|^2_{H^1_{\gamma}(Q^-)}\right)
 \leq \frac{C}{\gamma}\|\mathcal{F}_\gamma\|^2_{H^1_{\tan,\gamma}(Q^+)}
, \quad \gamma\ge \gamma_0,
\label{H11}
\end{equation}
where $C$ is independent of $\gamma$. Finally, from
\eqref{gradphi}, \eqref{interpol1} and \eqref{H11} we get
\begin{equation}
\begin{array}{ll}\label{gradphi2}
\ds \gamma \left( \|\mathcal{U}_{n\gamma}|_\omega\|^2_{H^{1/2}_{\gamma}(\omega)}+ \|\Ws_{\gamma}|_\omega\|^2_{H^{1/2}_{\gamma}(\omega)} \right)
+ \gamma^2 \|\varphi\|^2_{H^{1}_{\gamma}(\omega)}
 \leq \frac{C}{\gamma}\|\mathcal{F}_\gamma\|^2_{H^1_{\tan,\gamma}(Q^+)}
.
\end{array}
\end{equation}
Adding \eqref{H11}, \eqref{gradphi2} gives \eqref{54}; the proof of Theorem \ref{main} is complete.




\appendix

\section{Notation}\label{App:Notation}

In this Appendix, we collect the notation used in the paper, for the ease of reference.

\bigskip

In this paper, we deal with scalar and vector functions in the form~$f=f(t,x)$, where~$t\in[0,T]$ or~$t\in(-\infty,T]$, and $x=(x_1,x_2,x_3)\in \R^3$, or~$f=f(t,x')$, where~$x'=(x_2,x_3)\in\R^2$. By~$\p_j f$ we denote its partial derivative with respect to~$t$ if~$j=0$, or with respect to~$x_j$, if~$j=1,2,3$. We also put~$\p_t=\p_0$. If~$f$ is a scalar function, then we use notation~$f_j$ to denote~$f_j=\p_j f$, whereas if~$f$ is a vector function, we use notation~$f_k$ to denote the~$k$-th component of~$f$ (see Notation~\ref{not:derivatives}).

In the plasma region~$\Omega^+(t)$, we denote by~$U=(q,v,H,S)\tm$ the $8$-th dimensional vector function obtained by the total pressure~$q$, the $3$-dimensional plasma velocity~$v$, the $3$-dimensional magnetic field in plasma region~$H$ and the entropy~$S$. In the vacuum region~$\Omega^-(t)$, we denote by~$\mathcal{W}=(\Hc,\Ec)\tm$ the~$6$-th dimensional vector function obtained by the $3$-dimensional vacuum magnetic field~$\mathcal{H}$, and the $3$-dimensional vacuum electric field~$\mathcal{E}$. We assume that the interface between plasma and vacuum is given by a hypersurface~$\Gamma(t)$ which can be described by the equation $x_1=\varphi(t,x')$. Hence we may write
\[ \Omega^\pm(t)\doteq  \{ x\in\R^3: \ x_1 \gtrless \varphi(t,x') \}\,, \qquad \Gamma(t)\doteq  \{ (x_1,x') \in\R^3: \ x_1=\varphi(t,x') \}\,. \]
We write the systems for~$U$ and~$\Wc$ (see~\eqref{4} and~\eqref{eq:MaxwellB}) as:
\[ A_0(U )\partial_tU+\sum_{j=1}^3A_j(U )\partial_jU=0, \quad \text{in~$\Omega^+(t)$,} \qquad \mep \p_t \Wc + \sum_{j=1}^3 B_j\p_j \Wc = 0, \quad \text{in~$\Omega^-(t)$.} \]
We notice that (see~\eqref{eq:Bmatrix}):
\begin{gather*}B_j \doteq  \begin{pmatrix}
O_3 & B'_j \\
{B'_j}\tm & O_3
\end{pmatrix} ,\ j=1,2,3, \\
\quad B'_1\doteq \begin{pmatrix}
0 & 0 & 0 \\
0 & 0 & -1 \\
0 & 1 & 0
\end{pmatrix}, \
B'_2\doteq \begin{pmatrix}
0 & 0 & 1 \\
0 & 0 & 0 \\
-1 & 0 & 0
\end{pmatrix}, \
B'_3\doteq \begin{pmatrix}
0 & -1 & 0 \\
1 & 0 & 0 \\
0 & 0 & 0
\end{pmatrix}.\end{gather*}
After the change of coordinates introduced in Lemma~\ref{lemma3} by
\[ (t,\Phi(t,x)) = (t, x_1+\Psi(t,x){,x'})\,,\]
and dropping the tilde for convenience, we derive the new set of variables in the vacuum region (see Proposition~\ref{prop:Maxwell})
\begin{align*}
\hg & =( \Hc_1-\Psi_2 \Hc_2 -\Psi_3 \Hc_3,(1+\Psi_1) \Hc_2,(1+\Psi_1) \Hc_3)\tm, \\
\eg & =( \Ec_1-\Psi_2 \Ec_2-\Psi_3 \Ec_3,(1+\Psi_1)\Ec_2,(1+\Psi_1)\Ec_3)\tm,\\
\Hg & =((1+\Psi_1) \Hc_1, \Hc_2+\Psi_2 \Hc_1, \Hc_3+\Psi_3 \Hc_1)\tm,\\
\Eg & =((1+\Psi_1) \Ec_1, \Ec_2+\Psi_2 \Ec_1, \Ec_3+\Psi_3 \Ec_1)\tm,\\
\Hs & =((1+\Psi_1) \Hc_1, \Hc_2+\Psi_2 \Hc_1+\mep\Psi_t \Ec_3, \Hc_3+\Psi_3 \Hc_1-\mep\Psi_t \Ec_2)\tm\\
    & =\Hg +{\mep\Psi_t (0, \Ec_3,-\Ec_2)\tm,}\\
\Es & =((1+\Psi_1) \Ec_1, \Ec_2+\Psi_2 \Ec_1-\mep\Psi_t \Hc_3, \Ec_3+\Psi_3 \Ec_1+\mep\Psi_t \Hc_2)\tm\\
    & =\Eg+{\mep\Psi_t(0,-\Hc_3,\Hc_2)\tm.}
\end{align*}
We also denote
\[ \wg=(\hg,\eg)\tm, \qquad \Wg=(\Hg,\Eg)\tm, \qquad \Ws=(\Hs,\Es)\tm\,.\]
In the plasma region, we similarly define
\[ h = (H_1-\Psi_2 H_2 -\Psi_3 H_3, (1+\Psi_1) H_2,(1+\Psi_1) H_3)\tm\,. \]
We notice that
\[ \hg_1=\Hc_N \doteq  \Hc \cdot N\,, \qquad \eg_1=\Ec_N \doteq  \Ec\cdot N \,, \qquad h_1=H_N\doteq  H\cdot N\,, \]
where~$N\doteq  (1,-\Psi_2,-\Psi_3)\tm$. We also define~$v_N \doteq  v\cdot N$.
\\
The plasma and vacuum regions are now given by~$\Omega^\pm=\{x\in\R^3, \ x_1\gtrless 0\}$, whereas the boundary is given by~$\Gamma=\{ (0,x'), \ x'\in\R^2 \}$. We also define~$Q^\pm \doteq  (-\infty,T] \times \Omega^\pm$ and~$\omega_T\doteq  (-\infty,T]\times \Gamma$ (see Section~\ref{s2.2}).

\bigskip

To describe the relations between $\Wc, \wg, \Wg, \Ws$, we use the matrices
\[ \eta \doteq  \begin{pmatrix}
1 & - \Psi_2 & - \Psi_3 \\
0 & 1+\Psi_1 & 0 \\
0 & 0 & 1+\Psi_1
\end{pmatrix} 
\,, \qquad K = \begin{pmatrix}
\eta & 0 \\
0 & \eta
\end{pmatrix}\,, \]
\begin{align*} J & \doteq  \begin{pmatrix}
(1+\Psi_1)(\eta\tm)^{-1} & -\mep\Psi_t B_1'  \\
\mep\Psi_t B_1' & (1+\Psi_1)(\eta\tm)^{-1}
\end{pmatrix} \\
& \equiv \begin{pmatrix}
1+\Psi_1 & 0 & 0 & 0 & 0 & 0 \\
\Psi_2 & 1 & 0 & 0 & 0 & \mep\Psi_t \\
\Psi_3 & 0 & 1 & 0 & -\mep\Psi_t & 0 \\
0 & 0 & 0 & 1+\Psi_1 & 0 & 0  \\
0 & 0 & -\mep\Psi_t & \Psi_2 & 1 & 0  \\
0 & \mep\Psi_t & 0 & \Psi_3 & 0 & 1
\end{pmatrix} \,,\end{align*}
\[ L\doteq  \frac1{1+\Psi_1}\,J\,K\tm \, ,\] 
%
which give us (see~\eqref{eq:K} and~\eqref{eq:J})
\[ \wg=K\Wc\,, \qquad \Wg=\frac1{1+\Psi_1}\,K\tm \Wc\,, \qquad \Ws = J \,\Wc\,, \qquad \Ws=L\Wg\,.\]
We also define (see~\eqref{eq:J1})
\[ J_1 = (1+\Psi_1)(1-{\mep^2}\Psi_t^2)J^{-1} \,, \]
which is given by
{\tiny\[ J_1 \doteq  \begin{pmatrix}
1-{\mep^2}\Psi_t^2 & 0 & 0 & 0 & 0 & 0\\
-\Psi_2 & 1+\Psi_1 & 0 & \mep\,\Psi_3\,\Psi_t & 0 & -\mep\,(1+\Psi_1) \,\Psi_t\\
-\Psi_3 & 0 & 1+\Psi_1 & -\mep\,\Psi_2\,\Psi_t & \mep\,(1+\Psi_1) \,\Psi_t & 0\\
0 & 0 & 0 & 1-{\mep^2}\Psi_t^2 & 0 & 0\\
-\mep\,\Psi_3\,\Psi_t & 0 & \mep\,(1+\Psi_1) \,\Psi_t & -\Psi_2 & 1+\Psi_1 & 0\\
\mep\,\Psi_2\,\Psi_t & -\mep\,(1+\Psi_1) \,\Psi_t & 0 & -\Psi_3 & 0 & 1+\Psi_1
\end{pmatrix} \,.\]}
%
%
System~\eqref{eq:Maxwell1} may be written by means of variables~$\wg$ and~$\Wg$ (see Subsection \ref{subsecfixed}):
\[\begin{cases}
\mep\,\p_t \hg - \mep\,\frac{\Psi_t}{1+\Psi_1}\, \p_1\hg + \rot \Eg + \mep\,R'\hg = 0\,,\\
\mep\,\p_t \eg - \mep\,\frac{\Psi_t}{1+\Psi_1}\, \p_1\eg - \rot \Hg + \mep\,R'\eg = 0\,,
\end{cases}
\]
where
\[ R'= \begin{pmatrix}
0 & \p_2(\Psi_t/(1+\Psi_1)) & \p_3(\Psi_t/(1+\Psi_1)) \\
0	& -\p_1(\Psi_t/(1+\Psi_1)) & 0 \\
0 & 0 & -\p_1(\Psi_t/(1+\Psi_1))
\end{pmatrix}\,. \]

\bigskip

After we linearize the problem (see Section~\ref{sec:linearization}), we use a ``hat" to distinguish the variables corresponding to the \emph{basic state}, in particular $\hat U, \hat \Wc, \hat\varphi$ and the related quantities, like~$\hat\wg, \hat\Wg, \hat\Ws, \hat h, \hat N, \hat v_N, \hat\eta$ (see Section~\ref{s2.2}). On the other hand, in the construction of the composite quantities without the ``hat", like~$\wg,\Wg, \Ws, h, N, v_N$, the \emph{basic state} only appears in~$\hat\Psi$ (see Notation~\ref{not:hat}). We set
\begin{align*}
\hat\mu  \doteq  (\hat\Ec_1-\mep\,\hat v_3\hat\Hc_2+\mep\,\hat v_2\hat\Hc_3)_{|x_1=0}\,.
\end{align*}

In~\eqref{eq:B0}, we define
\[ B_0= KJ^{-1}\,, \]
the matrix which allow us to write the system for~$\Wc$ in the form
\[ {\mep\, B_0\p_t \Ws} + \sum_{{j=1}}^3 B_j\p_j \Ws {+\mep\,B_4\Ws}=0\,. \]
%

\section{Proofs of technical computations}\label{computations}

\begin{proof}[Proof of Lemma~\ref{lemma3}]
We define
\begin{equation}
\label{def-f1}
\Psi(t,x_1,x') \doteq \chi (x_1\langle D\rangle) \, \varphi(t,x') \, ,
\end{equation}
where $\chi\in C^\infty_0(\R)$ satisfies $\chi=1$ on $[-1,1]$, $0\leq\chi\leq1$, and $\chi(x_1 \langle D\rangle)$ is the pseudo-differential operator with $\langle D\rangle=(1+|D|^2)^{1/2}$ being the Fourier multiplier in the
variable $x'$. The proof of Lemma~\ref{lemma3} follows as in Lemmas 1, 2, 3 in~\cite{SeTr}, exception given for~\eqref{eq:Phitcontrol}, which is a new estimate that we will use in the following. To prove this latter, it is sufficient to notice that $\p_t^j\Phi(t,x)=\p_t^j\Psi(t,x) (1,0,0)\tm$ and to estimate
\[ \|\p_t^j\Psi(t,x_1,\cdot)\|_{L^\infty(\R^2)} \leq \|\langle D\rangle^{-\frac32}\chi (x_1\langle D\rangle) \|_{L^2(\R^2)} \, \|\p_t^j\varphi(t,\cdot)\|_{H^{\frac32}(\R^2)} \]
for any~$t\in[0,T]$ and for any~$x_1\in\R$. Indeed, by Plancherel's formula, being~$0\leq\chi\leq1$, the estimate
\begin{equation}\label{eq:estchi}\begin{split}
\|\langle D\rangle^{-\frac32}\chi (x_1\langle D\rangle) \|_{L^2(\R^2)} =& \frac1{2\pi} \|\langle \xi\rangle^{-\frac32}\chi(x_1\langle\xi\rangle)\|_{L^2(\R^2)} \\
\leq&  \frac1{2\pi} \|\langle \xi\rangle^{-\frac32}\|_{L^2(\R^2)} = \frac1{\sqrt{2\pi}}
\end{split}\end{equation}
concludes the proof of~\eqref{eq:Phitcontrol}.
\end{proof}
The following remark shows how the diffeomorphism in Lemma~\ref{lemma3} influences our differential system.
\begin{remark}\label{rem:derivatives}
Let~$f$ be a scalar function and let us denote by~$(t,\Phi(t,x))$ the change of coordinates $(t,x)\mapsto(t,\Phi)$. Since
\[ (f \circ (t,\Phi))_j = f_j \circ (t,\Phi) + \Psi_j \cdot f_1 \circ (t,\Phi) \qquad j=0,1,2,3, \]
if we put $\tilde{f}=f(t,\Phi)$, then it holds:
\begin{align*}
f_1 \circ (t,\Phi)
	& = \frac1{1+\Psi_1}\, \tilde{f}_1 \,,\\
f_j \circ (t,\Phi)
	& = \tilde{f}_j - \Psi_j \cdot f_1 \circ (t,\Phi) = \tilde{f}_j - \frac{\Psi_j}{1+\Psi_1} \, \tilde{f}_1\,, \qquad j\neq 1\,.
\end{align*}
\end{remark}
We may now prove Proposition~\ref{prop:Maxwell}.
\begin{proof}[Proof of Proposition~\ref{prop:Maxwell}]
Thanks to Remark~\ref{rem:derivatives},
for the second component of $\mep\p_t\Hc+\rot \Ec=0$ in~$\Omega^-(t)$ we have:
\begin{align*}
	&(1+\Psi_1)(\mep\p_t\Hc_2+\p_3\Ec_1-\p_1 \Ec_3) \circ (t,\Phi (t,x))\\
	& \qquad = \mep(1+\Psi_1)\p_t\tilde \Hc_2 -\mep\Psi_t \p_1\tilde\Hc_2 + (1+\Psi_1)\p_3\tilde\Ec_1 - \Psi_3\p_1\tilde\Ec_1 - \p_1\tilde\Ec_3\\
	& \qquad = \mep\p_t ((1+\Psi_1)\tilde\Hc_2) -\mep\p_1 (\Psi_t \tilde\Hc_2) + \p_3((1+\Psi_1)\tilde\Ec_1) - \p_1(\Psi_3\tilde\Ec_1) - \p_1\tilde\Ec_3 \\
	& \qquad = \mep\p_t \hg_2 + \p_3 \Es_1 -\p_1\Es_3\,,
\end{align*}
and similarly for the third one, and for the second and third components of $\mep\p_t\Ec-\rot \Hc=0$ in~$\Omega^-(t)$. To deal with the first component of $\mep\p_t\Hc+\rot \Ec=0$ in~$\Omega^-(t)$ we need more computation. First we notice that
\begin{align*}
& (1+\Psi_1)(\p_2 \Ec_3 - \p_3\Ec_2) \circ (t,\Phi (t,x)) \\
	& \qquad = (1+\Psi_1)\p_2\tilde\Ec_3 - \Psi_2\p_1\tilde\Ec_3 - (1+\Psi_1)\p_3\tilde\Ec_2 + \Psi_3\p_1\tilde\Ec_2   \\
	& \qquad  = \p_2 \Es_3 - \p_3 \Es_2 -\mep \p_2 (\Psi_t\tilde\Hc_2) -\mep \p_3 (\Psi_t\tilde\Hc_3) + \nabla\Psi \cdot (\rot \tilde\Ec)\, .
\end{align*}
We may use \eqref{eq:Maxwell} in~$\Omega^-(t)$ to derive
\[ \nabla\Psi \cdot (\rot \tilde\Ec) = \nabla\Psi \cdot \bigl( (\rot \Ec) \circ (t,\Phi) \bigr) = - \mep \nabla\Psi \cdot \bigl( (\p_t\Hc) \circ (t,\Phi) \bigr) \,,\] 
where the first equality follows from straightforward calculations. Therefore, replacing
\begin{align*}
	& (1+\Psi_1)(\p_t\Hc_1) \circ (t,\Phi) -\nabla\Psi \cdot \bigl((\p_t\Hc) \circ (t,\Phi)\bigr)\\
	& \qquad = \p_t\tilde\Hc_1 -\Psi_2 \p_t\tilde\Hc_2 -\Psi_3 \p_t\tilde\Hc_3 -\frac{\Psi_t}{1+\Psi_1} (\p_1\tilde\Hc_1-\Psi_2\p_1\tilde\Hc_2-\Psi_3\p_3\tilde\Hc_3) \\
	& \qquad = \p_t\hg_1 -\frac{\Psi_t}{1+\Psi_1} \p_1 \hg_1 + \Psi_{2t}\tilde\Hc_2+\Psi_{3t}\tilde\Hc_3 -\frac{\Psi_t}{1+\Psi_1} (\Psi_{12}\tilde\Hc_2+\Psi_{13}\tilde\Hc_3)\,,
\end{align*}
we obtain
\begin{align*}
& (1+\Psi_1)(\mep\p_t\Hc_1+\p_2 \Ec_3 - \p_3\Ec_2) \circ (t,\Phi (t,x)) \\
	& \qquad = \mep\p_t \hg_1 + \p_2 \Es_3 - \p_3 \Es_2 - \mep\,\frac{\Psi_t}{1+\Psi_1} \,\dv\hg .\end{align*}
We proceed similarly for the second equation of \eqref{eq:Maxwell} and this completes the proof of \eqref{eq:Maxwell1}.
Thanks again to Remark~\ref{rem:derivatives}, we obtain:
\begin{align*}
& (1+\Psi_1)(\dv \Hc)\circ (t,\Phi (t,x))\\
	& \qquad = \p_1 \, \tilde\Hc_1 +  (1+\Psi_1) \p_2\, \tilde\Hc_2 - \Psi_2\, \p_1 \tilde\Hc_2 + (1+\Psi_1) \p_3\, \tilde\Hc_3 - \Psi_3\, \p_1 \tilde\Hc_3 \\
	& \qquad = \p_1 \, (\tilde\Hc_1 - \Psi_2\, \tilde\Hc_2 - \Psi_3\, \tilde\Hc_3) + \p_2\, ( (1+\Psi_1)\tilde\Hc_2)  + \p_3\, ((1+\Psi_1)  \tilde\Hc_3)\,,
\end{align*}
which proves the equivalence of~$\dv\Hc=0$ in~$\Omega^-(t)$ and~$\dv\hg=0$ in the fixed domain. The equivalence of~$\dv\Ec=0$ in~$\Omega^-(t)$ and~$\dv\eg=0$ in the fixed domain is identical.
\end{proof}
\begin{proof}[Proof of Proposition~\ref{p1}]

The proof of the part concerning $\dv h, H_N$ can be found in \cite{trakhinin09arma}. Let us consider the vacuum variables. Taking the divergence of the first equation in \eqref{eq:Maxwell1} yields
\begin{equation*}
\begin{array}{ll}\label{}
\p_t \dv\hg- \p_1\left(\dfrac{\Psi_t}{1+\Psi_1}\dv\hg \right)=0.
\end{array}
\end{equation*}
Multiplying by $\dv\hg$ and integrating by parts gives
\begin{equation*}
\begin{array}{ll}\label{}
\dfrac{d\ }{dt}\int_{x_1<0}| \dv\hg|^2dx- \int_{x_1=0} \varphi_t|\dv\hg|^2dx'
-\int_{x_1<0}\p_1\left(\frac{\Psi_t}{1+\Psi_1} \right)| \dv\hg|^2dx=0.
\end{array}
\end{equation*}
Using the first of the equations in \eqref{eq:div2} we obtain
\begin{equation*}
\begin{array}{ll}\label{}
\dfrac{d\ }{dt}\int_{x_1<0}| \dv\hg|^2dx\leq C\int_{x_1<0}| \dv\hg|^2dx.
\end{array}
\end{equation*}
Then the Gronwall lemma and the assumption on the initial data yields $\dv\hg=0$ on $[0,T]\times \Omega^-$.
The proof of $\dv\eg=0$ is similar.
Considering now the first component of the first equation of {\eqref{eq:Maxwell1}} evaluated at $\Gamma$,
\begin{equation*}
\begin{array}{ll}\label{}
\mep \p_t \hg_1 + \p_2 \Es_3 - \p_3 \Es_2 - {\mep\varphi_t} \,\dv\hg=0,
\end{array}
\end{equation*}
the two last boundary conditions in {\eqref{17}}, namely $\Es_2=\Es_3=0$, and the result just obtained $\dv\hg=0$, gives $\hg_1(t,\cdot)=\hg_1(0,\cdot)=\Hc_N(0,\cdot)=0$ on $[0,T]\times\Gamma$.
\end{proof}

\section{On the number of boundary conditions for problem \eqref{16'}}\label{calcolosegni}
The correct number of boundary conditions that should be imposed to \eqref{16'} for well-posedness is given by the number of incoming characteristics, i.e. the number of negative eigenvalues of the boundary matrix of \eqref{16'}. Evaluated at $\Gamma$ (where $\Psi=\varphi,\Psi_1=0$) it reads
\begin{equation}\label{eq:B1tilde} \tilde{B}_1(\Psi)_{|x_1=0} =
- \begin{pmatrix}
\mep\varphi_t & 0 & 0 & 0 & -\varphi_3 & \varphi_2 \\
0 & \mep\varphi_t & 0 & \varphi_3 & 0 & 1 \\
0 & 0 & \mep\varphi_t & -\varphi_2 & -1 & 0 \\
0 & \varphi_3 & -\varphi_2 & \mep\varphi_t & 0 & 0  \\
-\varphi_3 & 0 & -1&0 & \mep\varphi_t & 0  \\
\varphi_2 & 1& 0 & 0 & 0 & \mep\varphi_t
\end{pmatrix} \,.
\end{equation}
This matrix has eigenvalues
\begin{equation*}
\begin{array}{ll}\label{}
\lambda_{1,2}=-\mep\varphi_t +\sqrt{1+\varphi_2^2+\varphi_3^2}\ , \qquad
\lambda_{3,4}=-\mep\varphi_t\ ,
\end{array}
\end{equation*}
\begin{equation*}
\begin{array}{ll}\label{}
\lambda_{5,6}=-\mep\varphi_t-\sqrt{1+\varphi_2^2+\varphi_3^2}\ .
\end{array}
\end{equation*}
In agreement with our choice of measure units with the speed of light in vacuum set to be~$1$, we may assume $\mep\,|\varphi_t|<1$. Thus we have $\lambda_{1,2}>0$ and $\lambda_{5,6}<0$. If $\varphi_t<0$ (the plasma expands into vacuum) $\lambda_{3,4}>0$ and so the total number of negative eigenvalues of $\tilde{B}_1(\Psi)_{|x_1=0}$ is two. Considering that the boundary matrix $\tilde{A}_1(U,\Psi)_{|x_1=0}$ has one negative and one positive eigenvalue, see Proposition \ref{nrbc}, and that one more boundary condition needs to be imposed for the determination of the front $\varphi$, the correct number of boundary conditions for the resolution of \eqref{16}, \eqref{16'} is four, as in \eqref{17}.

If $\varphi_t>0$ (the vacuum expands into plasma) $\lambda_{3,4}<0$ and so the total number of negative eigenvalues of $\tilde{B}_1(\Psi)_{|x_1=0}$ is four. Then the correct number of boundary conditions for the resolution of \eqref{16}, \eqref{16'} is six, as in \eqref{17}, \eqref{eq:div2}.

Differently, after the introduction of the new variables $\Ws$, the negative eigenvalues of the boundary matrix $B_1$ of \eqref{16'.1} are always two, see Proposition \ref{nrbc}, so that the correct number of boundary conditions for the resolution of \eqref{16.1}, \eqref{16'.1} is four, as in \eqref{17.1}.

\section*{Acknowledgments} We would like to warmly thank Yuri Trakhinin for pointing out a mistake in the first version of the paper, for many fruitful discussions and helpful suggestions
\footnote{After completing our paper we heard of the paper by Mandrik--Trakhinin \cite{mandrik-trakhinin} about the same problem.}.


\medskip
\medskip

\end{document}